\documentclass[10pt, a4paper, reqno, english]{amsart}

\usepackage[hidelinks]{hyperref}
\usepackage{enumerate}

\newcommand\AAA{\mathbb{A}}
\newcommand\CC{\mathbb{C}}
\newcommand\FF{\mathbb{F}}
\newcommand\PP{\mathbb{P}}
\newcommand\ZZ{\mathbb{Z}}
\newcommand\QQ{\mathbb{Q}}
\newcommand\RR{\mathbb{R}}
\newcommand\Qbar{{\overline{\QQ}}}
\newcommand\uu{\mathbf{u}}
\newcommand\vv{\mathbf{v}}
\newcommand\ww{\mathbf{w}}
\newcommand\xx{\mathbf{x}}
\newcommand\yy{\mathbf{y}}
\newcommand\zz{\mathbf{z}}

\newcommand\dd{\,\mathrm{d}}

\newcommand{\OO}{\mathcal{O}}
\newcommand{\Xm}{\mathfrak{X}}

\DeclareMathOperator\Pic{Pic}
\DeclareMathOperator\vol{vol}
\DeclareMathOperator\Gal{Gal}
\DeclareMathOperator\Symm{Sym}
\DeclareMathOperator\Hilbb{Hilb}
\newcommand\Sym[3]{\Symm^{#1}(\PP^{#2}_{#3})}
\newcommand\Hilb[3]{\Hilbb^{#1}(\PP^{#2}_{#3})}
\DeclareMathOperator\Proj{Proj}
\DeclareMathOperator\Spec{Spec}
\DeclareMathOperator\id{id}
\DeclareMathOperator\Char{char}
\newcommand\reg{\mathrm{reg}}
\newcommand\tX{{\widetilde{X}}}
\newcommand\tH{{\widetilde{H}}}

\newtheorem{theorem}{Theorem}
\newtheorem{lemma}[theorem]{Lemma}
\newtheorem{prop}[theorem]{Proposition}
\theoremstyle{definition}
\newtheorem{definition}[theorem]{Definition}
\newtheorem{remark}[theorem]{Remark}

\numberwithin{theorem}{section}
\numberwithin{equation}{section}

\begin{document}

\title[Manin's conjecture for the chordal cubic fourfold]
{Manin's conjecture for the\\chordal cubic fourfold}

\author{Ulrich Derenthal}

\address{Institut f\"ur Algebra, Zahlentheorie und Diskrete Mathematik, Leibniz Universit\"at Hannover, Welfengarten 1, 30167 Hannover, Germany}

\email{derenthal@math.uni-hannover.de}

\date{April 22, 2025}

\begin{abstract}
  We prove the thin set version of Manin's conjecture for the chordal
  (or: determinantal) cubic fourfold, which is the secant variety of
  the Veronese surface.  We reduce this counting problem to a result
  of Schmidt for quadratic points in the projective plane by showing
  that the chordal cubic fourfold is isomorphic to the symmetric
  square of the projective plane over the rational numbers.
\end{abstract}

\subjclass[2020]{11D25 (11D45, 14G05)}

%
%

\maketitle

\tableofcontents

\section{Introduction}

Manin's conjecture \cite{MR974910} predicts the asymptotic behavior of
the number of rational points of bounded height on (almost) Fano
varieties over number fields. The case of cubic hypersurfaces over
$\QQ$ might be the most intriguing one. Here, Manin's conjecture has
been proved so far only for singular cubic surfaces that are toric or
vector group compactifications \cite{MR1620682,MR1679843,MR3454090}
(using harmonic analysis), some singular cubic surfaces without such a
structure \cite{MR2332351,MR2520769,MR2990624,MR3263143,MR3414304}, a
singular cubic threefold \cite{MR2329549}, a singular cubic fourfold
\cite{MR3198752} (using universal torsors for all these
low-dimensional examples, combined with various analytic techniques),
and all smooth cubic hypersurfaces of dimension $\ge 15$
\cite{MR3605019} (using the circle method). In all these cases, one
obtains an asymptotic formula for the number of rational points of
bounded anticanonical height that agrees with the conjectures of
Manin, Peyre \cite{MR1340296,MR2019019} and Batyrev--Tschinkel
\cite{MR1679843} after removing a closed subset where the rational
points accumulate.

Here, we show that the thin set version \cite{MR2019019,MR4472281} of
the Manin--Peyre conjecture holds for (the minimal desingularization
of) the singular cubic fourfold $X \subset \PP^5_\QQ$ defined by the
cubic form
\begin{equation}\label{eq:cubic}
  P = X_0X_{12}^2+X_1X_{02}^2+X_2X_{01}^2-X_{01}X_{02}X_{12}-4X_0X_1X_2.
\end{equation}
Since
\begin{equation}\label{eq:determinant}
  -2P = \det
  \begin{pmatrix}
    2X_0 & X_{01} & X_{02} \\
    X_{01} & 2X_1 & X_{12} \\
    X_{02} & X_{12} & 2X_2
  \end{pmatrix}
  ,
\end{equation}
$X$ is the \emph{chordal cubic fourfold} or \emph{determinental cubic
  fourfold} described in \cite[\S 4.4]{MR1738215}. It is the secant
variety of the Veronese surface (the image of $\PP^2_\QQ$ under the
Veronese map \eqref{eq:diagonal_map}, which is the two-dimensional
Severi variety as classified in \cite{severi}, \cite{MR773432},
\cite[Theorem~4.7]{MR1234494}) and isomorphic to the symmetric square
of the projective plane (see \cite[Example~4.5]{MR1354269} over $\CC$
and our Proposition~\ref{prop:X_isom_Sym_2_2} over $\QQ$ and
$\ZZ$). The chordal cubic fourfold $X$ plays a special role in the GIT
compactification of the moduli space of cubic fourfolds \cite[\S
6]{MR2507640}, \cite[(2.8)]{MR2680429}, \cite[\S 6]{MR3959097},
\cite[Remark~4.8]{MR4680279}.

We will see (Lemma~\ref{lem:anticanonical_height}) that an
anticanonical height function on $X$ is defined by the third power of
exponential Weil height on $\PP^5(\QQ)$, namely
\begin{equation}\label{eq:height}
  H(\xx) := \max\{|x_0|,|x_1|,|x_2|,|x_{01}|,|x_{02}|,|x_{12}|\}^3
\end{equation}
for $\xx = (x_0:\dots:x_{12}) \in \PP^5(\QQ)$ represented by coprime
integers $x_0,\dots,x_{12}$.

The singular locus of $X$ is the Veronese surface. It is given by the
vanishing of the six $2\times 2$-minors of \eqref{eq:determinant}
(which are the partial derivatives of \eqref{eq:cubic}); to describe
it as a set, the three minors
\begin{equation*}
  \Delta_{ij} =  \frac{\partial P}{\partial X_k} = X_{ij}^2-4X_iX_j
\end{equation*}
(for $\{i,j,k\}=\{0,1,2\}$) are enough (see
Remark~\ref{rem:singular_locus}). 

For any subset $V \subset X(\QQ)$, we consider the counting function
\begin{equation}\label{eq:counting_function}
  N(V,H,B):=\{\xx \in V \mid H(\xx) \le B\}.
\end{equation}
For $\xx=(x_0:\dots:x_{12}) \in X(\QQ)$, the three numbers
$\Delta_{ij}(\xx) = x_{ij}^2-4x_ix_j \in \QQ$ are well-defined up to
multiplication by squares in $\QQ^\times$. Since they cannot have
different signs (see
Lemma~\ref{lem:X(k)_components}\eqref{it:delta_equal}), $X(\QQ)$ is
the disjoint union of
\begin{equation}\label{eq:def_X_0+-}
  \begin{aligned}
    X_0&=\{\xx \in X(\QQ) \colon \text{all $\Delta_{ij}(\xx)$ are squares in
         $\QQ$}\},\\
    X_+&=\{\xx \in X(\QQ) \colon \text{all $\Delta_{ij}(\xx) \ge 0$, but not all are
       squares in $\QQ$}\},\\
    X_-&=\{\xx \in X(\QQ) \colon \text{all $\Delta_{ij}(\xx) \le 0$, but not all vanish}\}.
  \end{aligned}
\end{equation}

\begin{theorem}\label{thm:main}
  Let $X \subset \PP^5_\QQ$ be the cubic fourfold defined by
  (\ref{eq:cubic}) with the anticanonical height function $H$ as in
  (\ref{eq:height}) and the subsets $X_0$, $X_+$, $X_-$ as in
  (\ref{eq:def_X_0+-}). Then there are positive real numbers $c_0$,
  $c_+$, $c_-$ such that
  \begin{align*}
    N(X_0,H,B) &= c_0 B\log B+O(B),\\
    N(X_+,H,B) &= c_+B\log B + O(B (\log B)^{1/2}),\\
    N(X_-,H,B) &= c_-B\log B + O(B (\log B)^{1/2}),
  \end{align*}
  as $B \to \infty$.
\end{theorem}

See Proposition~\ref{prop:asymptotic_formula} for the values of
$c_0,c_+,c_-$. To interpret this result as an instance of Manin's
conjecture, we must desingularize $X$. In the following result, the
counting function is defined analogously to
\eqref{eq:counting_function}.

\begin{theorem}\label{thm:manin-peyre}
  Let $h : \tX \to X$ be the blow-up of the singular locus, and let
  $\tH := H \circ h$. Then $\tX$ is an almost Fano variety as in
  \cite[D\'efinition~3.1]{MR2019019} that satisfies the thin set
  version \cite[5.2]{MR4472281} of Manin's conjecture with respect to
  the anticanonical height function $\tH$ on $\tX(\QQ)$. Here,
  $\tX_0:=h^{-1}(X_0)$ is the accumulating thin set predicted by
  Lehmann--Sengupta--Tanimoto, and in its complement, we have
  \begin{equation*}
    N(\tX(\QQ)\setminus\tX_0,\tH,B) = (c_++c_-)B \log B+O(B(\log B)^{1/2}),
  \end{equation*}
  where $c_++c_-$ Peyre's constant as in \cite[5.1]{MR2019019}.
\end{theorem}

See Remark~\ref{rem:exceptional_points_tX} for the rational points in
the accumulating thin set $\tX_0$.

In Section~\ref{sec:geometry}, we prove that $X$ is isomorphic to the
symmetric square of the projective plane over $\QQ$ and even over
$\ZZ$ since we could not find a reference working in this generality
(see \cite[Example~4.5]{MR1354269} over $\CC$). For this, we first
describe the symmetric square of the affine plane over $\ZZ$
(differently from its standard description in characteristic $0$
\cite[Example~7.1.3(2)]{MR2223408}, which does not work in
characteristic $2$), and then patch it together. (Note that the
factors $2$ on the diagonal of \eqref{eq:determinant} give a variety
that is isomorphic over $\QQ$ to the one described in \cite[\S
4.4]{MR1738215} without these factors, but they are important to
obtain an isomorphism to the symmetric square of the projective plane
over $\ZZ$ and for compatibility with Schmidt's results. Hence our
version of the chordal cubic fourfold seems to be the more natural
choice from an arithmetic point of view.)

In Section~\ref{sec:points}, we examine how the rational and local
points on $X$ are governed by the values of $\Delta_{ij}$ modulo
squares, depending on the ground field.

For the proof of Theorem~\ref{thm:main} in
Section~\ref{sec:asymptotic}, we derive from the isomorphism
$X \cong \Sym 2 2 \QQ$ that $X(\QQ)$ with our height function $H$ is
in bijection to the set of nonzero decomposable ternary quadratic
forms $f$ over $\QQ$, where proportional forms are counted as one,
with the cube of the height of the coefficient vector, which is
considered in work of Schmidt \cite[Theorem~4a]{MR1330740} .

For the proof of Theorem~\ref{thm:manin-peyre} in
Section~\ref{sec:expectation}, we compare the leading constants in
Schmidt's result to Peyre's constant for the desingularization
$\tX$, which is isomorphic to the Hilbert scheme $\Hilb 2 2 \QQ$ of
subschemes of length $2$ in $\PP^2_\QQ$. The main difficulty is to
compute the local densities, in particular at the places $v=2$ and
$v=\infty$.

Manin's conjecture for $\Sym 2 2 \QQ$ and its desingularization
$\Hilb 2 2 \QQ$ was also considered in work of Le Rudulier
\cite{LeRudulierThesis,MR4057715}. New in our work is the connection
to the chordal cubic fourfold. Furthermore, we work with a different
anticanonical height function on $\Sym 2 2 \QQ$ (the one used in
\cite{MR4057715} apparently cannot be easily expressed in terms of the
coordinates on $X$), hence we must apply a different result of Schmidt
(\cite[Theorem~4a]{MR1330740} instead of \cite[Theorem~3]{MR1330740}),
which has more complicated real densities in the leading
constants. Finally, we perform the comparison of the $p$-adic
densities in a different way (\cite[Proposition~3.2]{MR4057715}
compares the $p$-adic density to the number of points on the Hilbert
scheme over $\FF_p$; in view of \cite[\S 8.4.2,
Lemma~8.16]{MR3741845}, which mentions good reduction of the Hilbert
scheme only over odd primes, it seems preferable to us to compute the
expected $p$-adic densities by evaluating the corresponding $p$-adic
integral).

For a discussion of $\Hilb 2 2 \QQ$ in relation to Peyre's
\cite{MR3741845} notion of freeness, see \cite{sawin-freeness}.

\subsection*{Acknowledgements}

We learned about the connection between Schmidt's work and the
symmetric square of the projective plane from A.~Chambert-Loir and
M.~Widmer. We noticed the relation to the cubic form (\ref{eq:cubic})
while visiting the Centro di Ricerca Mate\-matica Ennio De Giorgi,
Pisa (2012), and continued to work on this while spending a sabbatical
at the University of Oxford (2018) visiting D.~R. Heath-Brown; their
hospitality is gratefully acknowledged. In O.~Debarre's talk at the
workshop \emph{Arithmetic and Geometry of Cubic Hypersurfaces} in
Hannover (2018), we learned that this is the chordal cubic
fourfold. We also thank E.~Peyre and P.~Salberger for helpful
discussions.

\section{Geometry}\label{sec:geometry}

In this section, we discuss the geometry of our singular cubic
fourfold $X$. The main result (Proposition~\ref{prop:X_isom_Sym_2_2})
shows that it is isomorphic to the symmetric square of the projective
plane, even over $\ZZ$.

To determine $\Sym 2 2 \ZZ$, we will glue copies of
$\Symm^2(\AAA^2_\ZZ)$. Over an algebraically closed field $K$ of
characteristic $0$, \cite[Example~7.1.3(2)]{MR2223408} shows that
\begin{equation}\label{eq:Sym_2_A^2_C}
  \Symm^2(\AAA^2_K) \cong \AAA^2_K \times
  \Spec(K[U_1',U_2',U_{12}']/(U_1'U_2'-U_{12}'^2)).
\end{equation}
In fact, the proof works over any field $K$ of characteristic $\ne 2$,
but the result is not true in characteristic $2$. Instead, we obtain
the following description in arbitrary characteristic and even over
$\ZZ$:

\begin{lemma}\label{lem:sym2a2}
  The symmetric square $\Symm^2(\AAA^2_\ZZ)$ of the affine plane over
  $\ZZ$ is isomorphic to
  \begin{equation*}
    \Spec(\ZZ[U_{01},U_{02},U_1,U_2,U_{12}]/(U_{12}^2+U_1U_{02}^2+U_2U_{01}^2-U_{01}U_{02}U_{12}-4U_1U_2)) \subset \AAA^5_\ZZ.
  \end{equation*}
  With $\tau \ne \id$ in $S_2$ acting on the coordinate ring $\ZZ[Y_1,Y_2,Z_1,Z_2]$ by
  exchanging $Y_i$ with $Z_i$, we have
  \begin{equation*}
    U_{0i}=Y_i+Z_i,\ U_i=Y_iZ_i,\ U_{12}=Y_1Z_2+Y_2Z_1 \in \ZZ[Y_1,Y_2,Z_1,Z_2]^{S_2}
  \end{equation*}
\end{lemma}

\begin{proof}
  By \cite[Proposition~V.1.1]{MR217087},
  $\Symm^2(\AAA^2_\ZZ) \cong \Spec(\ZZ[Y_1,Y_2,Z_1,Z_2]^{S_2})$. The
  invariant ring $\ZZ[Y_1,Y_2,Z_1,Z_2]^{S_2}$ contains
  $U_{01},U_{02},U_1,U_2,U_{12}$. By induction over the degree, we
  show that every homogeneous $P \in \ZZ[Y_1,Y_2,Z_1,Z_2]^{S_2}$ is in
  the $\ZZ$-subalgebra $A$ generated by these five polynomials. In
  degree $d=0$, there is nothing to prove. In degree $d>0$, a
  polynomial invariant under $\tau$ is an integral linear combination
  of invariant monomials (which clearly are products of powers of
  $Y_1Z_1=U_1$ and $Y_2Z_2=U_2$, hence in $A$) and of expressions of
  the form $M+\tau(M)$, where $M$ is a noninvariant monomial.

  To show that $M+\tau(M) \in A$ for all such $M$, we first consider
  the case that $M = Y_iZ_i M'$ for some monomial $M'$ (which must
  have lower degree); in this case $M+\tau(M) = U_i (M'+\tau(M'))$,
  where $M'+\tau(M') \in A$ by induction. Hence it remains to consider
  $M=Y_i^aZ_j^b$ for $\{i,j\}=\{1,2\}$, where we may
  assume $a\ge b$. It is not hard to see that
  $U_{0i}^{a-b}U_{12}^b=(Y_i+Z_i)^{a-b}(Y_iZ_j+Y_jZ_i)^b \in A$ is the
  sum of $M+\tau(M)$ and a multiple of $Y_iZ_i$, which is in $A$ by
  induction. Hence also $M+\tau(M) \in A$.

  Over the complex numbers, \cite[Example~7.1.3(2)]{MR2223408} shows that
  \begin{equation*}
    \Symm^2(\AAA^2_\CC) \cong \Spec(\CC[U_{01}',U_{02}',U_1',U_2',U_{12}']/(U_{12}'^2-U_1'U_2')),
  \end{equation*}
  as in (\ref{eq:Sym_2_A^2_C}), where
  \begin{equation*}
    U_{0i}' = Y_i+Z_i,\ U_i' = (Y_i-Z_i)^2,\ U_{12}'=(Y_1-Z_1)(Y_2-Z_2).
  \end{equation*}
  Clearly
  \begin{equation*}
    U_{0i}'=U_{0i},\ U_i' = U_{0i}^2-4U_i,\ U_{12}'=U_{01}U_{02}-2U_{12},
  \end{equation*}
  hence the relation is
  \begin{equation*}
    (U_{12}'^2-U_1'U_2')/4 = U_{12}^2+U_1U_{02}^2+U_2U_{01}^2-U_{01}U_{02}U_{12}-4U_1U_2.
  \end{equation*}
  Therefore, the ideal of relations between our five generators is
  generated by this polynomial over $\CC$, and the same is true over
  $\ZZ$ since any additional relation over $\ZZ$ would also hold over
  $\CC$.
\end{proof}

\begin{prop}\label{prop:X_isom_Sym_2_2}
  The map $\PP^2_\ZZ \times_\ZZ \PP^2_\ZZ \to \PP^5_\ZZ$ sending $((y_0:y_1:y_2),(z_0:z_1:z_2))$ to
  \begin{equation*}
    (y_0z_0:y_1z_1:y_2z_2:y_0z_1+y_1z_0:y_0z_2+y_2z_0:y_1z_2+y_2z_1).
  \end{equation*}
  induces an isomorphism between the symmetric square $\Sym 2 2 \ZZ$
  of the projective plane over $\ZZ$ and the chordal cubic fourfold
  \begin{equation*}
    \Xm = \Proj(\ZZ[X_0,X_1,X_2,X_{01},X_{02},X_{12}]/(P)) \subset \PP^5_\ZZ
  \end{equation*}
  over $\ZZ$ defined by $P \in \ZZ[X_0,X_1,X_2,X_{01},X_{02},X_{12}]$
  as in (\ref{eq:cubic}).
\end{prop}

\begin{proof}
  The scheme $\Sym 2 2 \ZZ$ is the quotient of
  $Y=\PP^2_\ZZ \times_\ZZ \PP^2_\ZZ$ by the action of the symmetric group
  $S_2=\{\id, \tau\}$ where $\tau$ swaps the components. The quotient
  exists since $S_2$ is finite and $Y$ is projective. We
  construct it as follows. We cover $Y$ by the six $S_2$-invariant affine open
  subsets
  \begin{equation*}
    W_i=\{Y_iZ_i \ne 0\},\quad W_{ij} = \{(Y_i+Y_j)(Z_i+Z_j) \ne 0\}.
  \end{equation*}
  Each of them is isomorphic to $\AAA^2_\ZZ \times_\ZZ \AAA^2_\ZZ$ with $S_2$ acting by
  swapping components, whose quotient is computed in
  Lemma~\ref{lem:sym2a2}.

  We claim that $W_i$ and $W_{ij}$ are isomorphic to the open subsets
  $V_i=\{X_i \ne 0\}$ and $V_{ij}=\{X_{ij}+X_i+X_j \ne 0\}$ of
  $X$, respectively. We note that these cover $X$. (We think of the variables on $X$
  as $X_i=Y_iZ_i$ and $X_{ij}=Y_iZ_j+Y_jZ_i$.)

  Indeed, an isomorphism from
  \begin{equation*}
    V_0 = \Spec(\ZZ[X_1,X_2,X_{01},X_{02},X_{12}]/(X_{12}^2+X_1X_{02}^2+X_2X_{01}^2-X_{01}X_{02}X_{12}-4X_1X_2))
  \end{equation*}
  to
  \begin{equation*}
    W_0/S_2 = \Spec(\ZZ[U_1,U_2,U_{01},U_{02},U_{12}]/(U_{12}^2+U_1U_{02}^2+U_2U_{01}^2-U_{01}U_{02}U_{12}-4U_1U_2))
  \end{equation*}
  is clearly induced by
  \begin{equation}\label{eq:map-V0-W0/S2}
    U_i \mapsto X_i,\ U_{ij} \mapsto X_{ij}.
  \end{equation}
  Analogously, we have $W_1\cong V_1$ and $W_2\cong V_2$.

  For $W_{01}/S_2 \cong V_{01}$, we use the isomorphism
  \begin{equation*}
    \AAA^2_\ZZ \times_\ZZ \AAA^2_\ZZ \to W_{01}, \quad ((y_1,y_2),(z_1,z_2)) \mapsto ((1-y_1:y_1:y_2),(1-z_1:z_1:z_2)).
  \end{equation*}
  to obtain an isomorphism $\Symm^2(\AAA^2_\ZZ) \to W_{01}/S_2$. The
  isomorphism
  \begin{equation*}
    \Symm^2(\AAA^2_\ZZ) \cong W_{01}/S_2 \to V_{01}=\Spec(\ZZ[X_0', \dots, X_{12}']/(P, X_{01}'+X_0'+X_1'-1)
  \end{equation*}
  is therefore induced by (obtained by setting $Y_0'=1-Y_1'$ and
  $Z_0'=1-Z_1'$ in the expressions for $X_i',X_{ij}'$ and expressing them
  in terms of $U_i',U_{ij}'$)
  \begin{align*}
    &X_0'\mapsto 1-U_{01}'+U_1',\ X_1' \mapsto U_1',\ X_2' \mapsto U_2', \\
    &X_{01}'\mapsto U_{01}'-2U_1',\ X_{02}' \mapsto U_{02}'-U_{12}',\ X_{12}' \mapsto U_{12}',
  \end{align*}
  with inverse $V_{01} \to W_{01}/S_2$ induced by
  \begin{equation}\label{eq:map-V01-W01/S2}
    U_1'\mapsto X_1',\ U_2'\mapsto X_2',\ U_{01}'\mapsto X_{01}'+2X_1',\ U_{02}' \mapsto X_{02}'+X_{12}',\ U_{12}'\mapsto X_{12}'.
  \end{equation}
  Again it is straightforward to check that the maps are inverse to
  each other and are compatible with the defining equations;
  $V_{02} \cong \Symm^2(\AAA^2_k) \cong V_{12}$ are analogous.

  Finally, we must show that these isomorphisms are compatible with
  the natural gluing maps between $W_i,W_{ij}$, and
  $W_i/S_2,W_{ij}/S_2$, and $V_i,V_{ij}$, respectively.

  For example, consider gluing $W_0$ and $W_{01}$ along
  \begin{equation*}
    W_0 \cap W_{01}=\{Y_0Z_0(Y_0+Y_1)(Z_0+Z_1) \ne 0\}.
  \end{equation*}
  Here,
  \begin{equation*}
    W_0=\Spec(k[Y_1,Y_2,Z_1,Z_2]), \qquad W_{01}=\Spec(k[Y_1',Y_2',Z_1',Z_2'])
  \end{equation*}
  contain principal open subsets $\{((1+Y_1)(1+Z_1) \ne 0\}$,
  $\{(1-Y_1')(1-Z_1') \ne 0\}$, respectively, which are isomorphic via
  \begin{equation*}
    ((1:y_1:y_2),(1:z_1:z_2)) = \left(\left(\tfrac{1}{1+y_1}:\tfrac{y_1}{1+y_1}:\tfrac{y_2}{1+y_1}\right),\left(\tfrac{1}{1+z_1}:\tfrac{z_1}{1+z_1}:\tfrac{z_2}{1+z_1}\right)\right),
  \end{equation*}
  i.e., via the map
  \begin{equation*}
    Y_1'\mapsto \tfrac{Y_1}{1+Y_1},\ Y_2'\mapsto \tfrac{Y_2}{1+Y_1},
    \ Z_1' \mapsto \tfrac{Z_1}{1+Z_1},\ Z_2'\mapsto \tfrac{Z_2}{1+Z_1}.
  \end{equation*}
  The induced isomorphism between the corresponding principal open
  subsets of $W_0/S_2$ (with coordinate functions
  $U_{01},U_{02},U_1,U_2,U_{12}$) and $W_{01}/S_2$ (with coordinates
  $U_{01}',U_{12}',U_1',U_2',U_{12}'$)
  is given by
  \begin{equation}\label{eq:map-W0/S2-W01/S2}
    \begin{aligned}
      &U_{01}' = Y_1'+Z_1' \mapsto \tfrac{Y_1}{1+Y_1}+\tfrac{Z_1}{1+Z_1}=\tfrac{Y_1+Z_1+2Y_1Z_1}{(1+Y_1)(1+Z_1)} = \tfrac{U_{01}+2U_1}{1+U_{01}+U_1},\\
      &U_{02}' \mapsto \tfrac{U_{02}+U_{12}}{1+U_{01}+U_1},\ U_1' \mapsto \tfrac{U_1}{1+U_{01}+U_1},\ U_2'\mapsto \tfrac{U_2}{1+U_{01}+U_1},\ U_{12}'\mapsto \tfrac{U_{12}}{1+U_{01}+U_1}.
    \end{aligned}
  \end{equation}
  We must compare this to the coordinate change between $V_0$ (with
  coordinate functions $X_i,X_{ij}$) and $V_{01}$ (with coordinate
  functions $X_i',X_{ij}'$) given by
  \begin{align*}
    &(1:x_1:x_2:x_{01}:x_{02}:x_{12})\\&= \left(\tfrac{1}{1+x_1+x_{01}}:\tfrac{x_1}{1+x_1+x_{01}}:\tfrac{x_2}{1+x_1+x_{01}}:\tfrac{x_{01}}{1+x_1+x_{01}}:\tfrac{x_{02}}{1+x_1+x_{01}}:\tfrac{x_{12}}{1+x_1+x_{01}}\right),
  \end{align*}
  which is induced by
  \begin{equation}\label{eq:map-V0-V01}
    X_0' \mapsto \tfrac{1}{1+X_1+X_{01}},\ X_1' \mapsto \tfrac{X_1}{1+X_1+X_{01}},\ \dots,\ X_{12}' \mapsto \tfrac{X_{12}}{1+X_1+X_{01}}.
  \end{equation}
  In total, the isomorphism $V_0 \to W_0/S_2 \to W_{01}/S_2$ is
  induced by the composition of \eqref{eq:map-W0/S2-W01/S2} and \eqref{eq:map-V0-W0/S2}, which maps
  \begin{equation*}
    U_{01}' \mapsto \tfrac{U_{01}+2U_1}{1+U_{01}+U_1} \mapsto \tfrac{X_{01}+2X_1}{1+X_{01}+X_1},
  \end{equation*}
  while $V_0 \to V_{01} \to W_{01}/S_2$ is induced by the composition
  of \eqref{eq:map-V01-W01/S2} and \eqref{eq:map-V0-V01}, which maps
  \begin{equation*}
    U_{01}' \mapsto X_{01}'+2X_1' \mapsto \tfrac{X_{01}+2X_1}{1+X_1+X_{01}};
  \end{equation*}
  similarly, we check that they agree on the other coordinates.
  
  The remaining computations are long, but straightforward.
\end{proof}

\begin{remark}\label{rem:singular_locus}
  As mentioned in the introduction, the singular locus of the chordal
  cubic fourfold $X \subset \PP^5_K$ is given by the six
  $2\times 2$-minors of the determinant in \eqref{eq:determinant},
  which are the partial derivaties
  \begin{equation*}
    \Delta_{01},\ \Delta_{02},\ \Delta_{12},\ X_{02}X_{12}-2X_2X_{01},\ X_{01}X_{12}-2X_1X_{02},\ X_{01}X_{02}-2X_0X_{12}
  \end{equation*}
  of the cubic form (\ref{eq:cubic}).  It is a reduced scheme for
  $\Char(K) \ne 2$, but nonreduced in characteristic $2$.

  As a set, the vanishing locus of these partial derivatives coincides
  with the subset of $X$ where $\Delta_{01},\Delta_{02},\Delta_{12}$
  vanish. Indeed, over a field $K$ of characteristic $\ne 2$, let
  $\xx = (x_0,\dots,x_{12}) \in X(K)$ with
  $\Delta_{01}(\xx)=\Delta_{02}(\xx)=\Delta_{12}(\xx)=0$. At least one
  of $x_0,x_1,x_2$ most be nonzero since otherwise $\Delta_{ij}=0$
  implies that also all $x_{ij}=0$. By symmetry, we may assume
  $x_0=1$. Then $x_1=x_{01}^2/4$, $x_2=x_{02}^2/4$, and hence
  $x_{12}^2=4x_1x_2=(x_{01}x_{02}/2)^2$, so that
  $x_{12} = \pm x_{01}x_{02}/2$. However, only the sign choice
  $x_{12} = x_{01}x_{02}/2$ satisfies (\ref{eq:cubic}), and then also
  the other three partial derivatives vanish. In characteristic $2$,
  it is easy to check that precisely the points with
  $x_{01}=x_{02}=x_{12}=0$ are singular, which is equivalent to
  $\Delta_{01}(\xx)=\Delta_{02}(\xx)=\Delta_{12}(\xx)=0$ in this case.

  However, $\Delta_{01},\Delta_{02},\Delta_{12}$ define a nonreduced
  scheme in any characteristic.

  The singular locus of $\Sym 2 2 K$ is the diagonal image of
  $\PP^2_K$ \cite[\S 2]{MR0335512}. Of course its image in $X$ under
  the isomorphism from Proposition~\ref{prop:X_isom_Sym_2_2} is the
  singular locus described above. Explicitly, it is obtained by the
  Veronese map 
  \begin{equation}\label{eq:diagonal_map}
    \PP^2_K\to X \subset \PP^5_K, \qquad
    (x_0:x_1:x_2)\mapsto (x_0^2:x_1^2:x_2^2:2x_0x_1:2x_0x_2:2x_1x_2),
  \end{equation}
  whose image is the Veronese surface.
\end{remark}

\section{Parameterization of rational and local points}\label{sec:points}

In this section, we show how the set of rational points on the chordal
cubic fourfold $X$ can be parameterized, in particular over the
completions $\QQ_v$ of $\QQ$.

\begin{definition}\label{def:R_k}
  Given a field $K$, let $R_K$ be a set of representatives for
  $K^\times/K^{\times 2}$; we choose $1$ to represent the trivial
  class $K^{\times 2}$. In particular, we make the following choices
  for $\QQ$ and its completions (see \cite[\S II.3.3]{MR344216}):
  \begin{itemize}
  \item Let $R_\QQ$ be the set of squarefree integers.
  \item Let $R_\RR=\{1,-1\}$.
  \item Let $R_{\QQ_2} = \{1,3,5,7,2,6,10,14\}$.
  \item For $p\ne 2$, let $R_{\QQ_p}=\{1,p,u,up\}$ for some 
    $u \in \ZZ_p^\times$ that reduces to a quadratic nonresidue modulo $p$.
  \end{itemize} 
\end{definition}

Every $K$-rational point on $\Sym 2 2 K$ is an unordered pair
$[\yy,\zz]$ of $K$-rational points (i.e., $\yy,\zz \in \PP^2(K)$) or
of conjugate points defined over a quadratic extension $K'$ of $K$
(i.e., $\yy,\zz \in \PP^2(K')$ with $\zz=\overline{\yy}$, where
$\overline{\,\cdot\,}$ is the nontrivial element of the Galois group
of $K'/K$). The quadratic extensions of $K$ are $K(\sqrt\epsilon)$ for
$\epsilon \in R_K$.

\begin{lemma}\label{lem:X(k)_components}
  Let $K$ be a field, with $R_K$ as in Definition~\ref{def:R_k}. Let
  $X \subset \PP^5_K$ be the chordal cubic fourfold defined by
  \eqref{eq:cubic} over $K$. For $\xx=(x_0: \dots: x_{12}) \in X(K)$
  and distinct $i,j \in \{0,1,2\}$, let
  $\Delta_{ij}(\xx):=x_{ij}^2-4x_ix_j$ (which is defined in $K$ up to
  multiplication by squares in $K^\times$).

  \begin{enumerate}[(i)]
  \item\label{it:delta_equal} Let $\{i,j,k\} = \{0,1,2\}$. If both $\Delta_{ik}(\xx)$ and
    $\Delta_{jk}(\xx)$ are nonzero, then they have the same class
    in $K^\times/K^{\times 2}$.
  \item\label{it:X(k)_disjoint_union} For $\epsilon \in R_K$, let
    \begin{equation*}
      X(K)^{(\epsilon)}:=\{\xx \in X(K) \mid \Delta_{01}(\xx),\Delta_{02}(\xx),\Delta_{12}(\xx) \in \epsilon K^{\times 2} \cup \{0\}\}.
    \end{equation*}
    Then
    \begin{equation*}
      X(K) = \bigcup_{\epsilon \in R_K} X(K)^{(\epsilon)}.
    \end{equation*}
    For distinct $\epsilon,\epsilon' \in R_K$, we have
    \begin{equation*}
      X(K)^{(\epsilon)} \cap X(K)^{(\epsilon')} = X(K)_\Delta:=\{\xx \in X(K) \mid \Delta_{01}(\xx)=\Delta_{02}(\xx)=\Delta_{12}(\xx)=0\}.
    \end{equation*}
  \item\label{it:Sym2_disjoint_union} Let $\xx \in X(K)$ be the image
    of $[\yy,\zz] \in \Sym 2 2 K (K)$ under the isomorphism from
    Proposition~\ref{prop:X_isom_Sym_2_2}. We have
    $\xx \in X(K)^{(1)}$ if and only if $\yy,\zz \in \PP^2(K)$. For
    $\epsilon \ne 1$, we have $\xx \in X(K)^{(\epsilon)}$ if and only
    if $\zz = \overline{\yy} \in \PP^2(K(\sqrt{\epsilon}))$. For
    distinct $\epsilon,\epsilon' \in R_K$, we have
    $\xx \in X(K)^{(\epsilon)} \cap X(K)^{(\epsilon')}$ if and only if
    $\yy = \zz \in \PP^2(K)$.
  \end{enumerate}
\end{lemma}

\begin{proof}
  To prove \eqref{it:delta_equal}, regarding (\ref{eq:cubic}) as a
  quadratic polynomial in $X_{ij}$, we observe that its discriminant
  is
  \begin{equation*}
    (-X_{ik}X_{jk})^2-4X_k(X_iX_{jk}^2+X_jX_{ik}^2-4X_0X_1X_2)=\Delta_{ik}\Delta_{jk}.
  \end{equation*}
  For $\xx \in X(K)$ over a field $K$, this discriminant must be a
  square in $K$, and hence $\Delta_{ik}(\xx)$ and $\Delta_{jk}(\xx)$
  must have the same class in $K^\times/K^{\times 2}$ if both are
  nonzero. More explicitly, we observe using \eqref{eq:cubic} that
  \begin{equation}\label{eq:Delta_square}
    \Delta_{ik}\Delta_{jk} = (X_{ik}X_{jk}-2X_{ij}X_k)^2.
  \end{equation}

  The statements in \eqref{it:X(k)_disjoint_union} follow directly
  from \eqref{it:delta_equal}.
  
  For any distinct $i,j \in \{0,1,2\}$, we have an isomorphism
  $\Symm^2(\PP^1_K) \to \PP^2_K$ given by
  \begin{equation*}
    [(y_i:y_j),(z_i:z_j)] \mapsto (x_i:x_{ij}:x_j):=(y_iz_i:y_iz_j+y_jz_i:y_jz_j).
  \end{equation*}
  It corresponds to the polynomial equation
  \begin{equation*}
    (y_iX+y_j)\cdot(z_iX+z_j)=y_iz_iX^2+(y_iz_j+y_jz_i)X+y_jz_j;
  \end{equation*}
  here, $(x_i:x_{ij}:x_j) \in \PP^2(K)$ if and only if
  $[(y_i:y_j),(z_i:z_j)]$ is an unordered pair of points over $K$ or
  of conjugate points over a quadratic extension
  $K(\sqrt{\epsilon})$. Given $(x_i:x_{ij}:x_j) \in \PP^2(K)$, let
  $\Delta_{ij} = x_{ij}^2-4x_ix_j \in K$ be the discriminant of our
  quadratic equation. We can recover $[(y_i:y_j),(z_i:z_j)]$ by
  determining the roots of $x_iX^2+x_{ij}X+x_j$; we obtain a pair of
  points over $K$ if and only if
  $\Delta_{ij} \in K^{\times 2} \cup \{0\}$, and a pair of points
  conjugate over $K(\sqrt{\epsilon})$ if and only if
  $\Delta_{ij} \in \epsilon K^{\times 2} \cup \{0\}$.

  Together, we obtain $[\yy,\zz] = [(y_0:y_1:y_2),(z_0:z_1:z_2)]$,
  where $\yy,\zz$ are either both defined over $K$ or conjugate to
  each other over $K(\sqrt{\epsilon})$ if we make the choices for the
  three pairs of indices $i,j$ in a compatible way. Here, the most
  interesting case is $x_0=y_0=z_0=1$, where we obtain
  \begin{equation*}
    y_j=\frac{x_{0j}+\sqrt{\Delta_{0j}}}{2},\quad 
    z_j=\frac{x_{0j}-\sqrt{\Delta_{0j}}}{2}
  \end{equation*}
  for $j \in \{1,2\}$. Considering $x_{12}=y_1z_2+y_2z_1$ shows that
  we must choose $\sqrt{\Delta_{01}}$ and $\sqrt{\Delta_{02}}$ (which
  cannot lead to different quadratic extensions of $K$ because of
  \eqref{it:delta_equal}) in a compatible way as follows:
  $\sqrt{\Delta_{01}}\sqrt{\Delta_{02}} = x_{01}x_{02}-2x_{12}$. The
  other cases are similar or easier.

  Every $\xx \in X(K)$ is the image of $[\yy,\zz]$ with
  $\yy,\zz \in \PP^2(K)$ or
  $\zz = \overline{\yy} \in \PP^2(K(\sqrt{\epsilon}))$ over some
  quadratic extension $K(\sqrt{\epsilon})$ of $K$. In the first case,
  our discussion above together with \eqref{it:delta_equal} shows that
  all $\Delta_{ij}(\xx) \in K^{\times 2} \cup \{0\}$; in the second
  case, it shows that all
  $\Delta_{ij}(\xx) \in \epsilon K^{\times 2} \cup \{0\}$. The
  intersection of both cases is clearly $\yy=\zz \in \PP^2(K)$ (which
  is the diagonal of $\Sym 2 2 K$), which is equivalent to
  $\Delta_{ij}(\xx) = 0$ for all $i,j$.
\end{proof}

\begin{remark}\label{rem:compare_notation_Q}
  Comparing (\ref{eq:def_X_0+-}) with the notation in
  Lemma~\ref{lem:X(k)_components}, we observe that
  \begin{itemize}
  \item $X_0 = X(\QQ)^{(1)}$,
   \item $X_+$ is the disjoint union of
    $X(\QQ)^{(d)} \setminus X(\QQ)_\Delta$ for positive squarefree
    $d \ne 1$,
  \item $X_-$ is the disjoint union of these sets for negative
    squarefree $d$.
 \end{itemize}
\end{remark}

\begin{lemma}\label{lem:openness}
  Consider $v \in \Omega_\QQ$ (the set of places of $\QQ$). For
  $\epsilon \in R_{\QQ_v}$, the set
  $X(\QQ_v)^{(\epsilon)} \setminus X(\QQ_v)_\Delta$ is an open subset
  of $X(\QQ_v)$ in the $v$-adic topology.
\end{lemma}

\begin{proof}
  Let $\xx \in X(\QQ_v)^{(\epsilon)} \setminus X(K)_\Delta$. Since
  $\QQ_v^{\times 2}$ is open in $\QQ_v$ \cite[\S II.3.3]{MR344216},
  the same holds for $\epsilon \QQ_v^{\times 2}$. Therefore, if
  $\Delta_{ij}(\xx) \in \epsilon \QQ_v^{\times 2}$ (which holds for at
  least one pair of indices $i,j$ outside $X(\QQ_v)_\Delta$), then any
  $\xx' \in X(\QQ_v)$ sufficiently close to $\xx$ also satisfies
  $\Delta_{ij}(\xx') \in \epsilon \QQ_v^{\times 2}$. If
  $\Delta_{ik}(\xx)=0$, then we have either $\Delta_{ik}(\xx')=0$ or
  $\Delta_{ik}(\xx')=\epsilon \QQ_v^{\times 2}$ (by
  Lemma~\ref{lem:X(k)_components}\eqref{it:delta_equal} using
  $\Delta_{ij}(\xx') \in \epsilon \QQ_v^{\times 2}$). In total,
  $\xx' \in X(\QQ_v)^{(\epsilon)} \setminus X(K)_\Delta$.
\end{proof}

Let $v \in \Omega_\QQ$. For any $x \in \QQ_v^{\times 2}$, we can choose
$\sqrt{x}$ among the two numbers in $\QQ_v$ (differing by sign) whose
square is $x$ as follows:
\begin{itemize}
\item For $v=\infty$, we choose $\sqrt{x} \in M_\infty:=\RR_{>0}$.
\item For $v=p \ne 2$, we choose
  \begin{equation*}
    \sqrt{x} \in M_p:=\left\{z \in \QQ_p^\times : |z|_p^{-1}z \equiv a \pmod p\text{ with }a \in \{1,\dots,\tfrac{p-1}2\}\right\}.
  \end{equation*}
\item For $v=2$, we choose $\sqrt{x} \in M_2:=\{z \in \QQ_2^\times : |z|_2^{-1}z \equiv 1 \pmod 4\}$.
\end{itemize}
In other words, we choose $\sqrt{x} \in \QQ_p$ such that the first
nontrivial digit in its $p$-adic expansion is as small as possible in
$\{0,\dots,p-1\}$. In any case, $M_v$ is an open subset of $\QQ_v$.

\begin{lemma}\label{lem:C1-parameterization-1}
  Consider the Zariski-open subset
  \begin{equation}\label{eq:U_0}
    U_0 = \{X_0 \ne 0,\ \Delta_{01} \ne 0,\ \Delta_{02} \ne 0\}
  \end{equation}
  of $X$. The sets
  \begin{equation*}
    U_0^{(1)}:=U_0(\QQ_v) \cap X(\QQ_v)^{(1)},\ Z_0^{(1)}:=\{(y_1,y_2,z_1,z_2) \in \QQ_v^{\oplus 4} : y_1 \ne z_1,\ y_2-z_2 \in M_v\}
  \end{equation*}
  are open subsets of $X(\QQ_v)$, $\QQ_v^{\oplus 4}$, respectively, in
  the $v$-adic topology. The map
  \begin{equation*}
    (y_1,y_2,z_1,z_2) \mapsto (1:y_1z_1:y_2z_2:y_1+z_1:y_2+z_2:y_1z_2+y_2z_1)
  \end{equation*}
  defines a bijection $Z_0^{(1)} \to U_0^{(1)}$ with inverse
  \begin{equation*}
    (1:x_1:x_2:x_{01}:x_{02}:x_{12}) \mapsto \left(\tfrac{x_{01}+\sqrt{\Delta_{01}}}{2},\tfrac{x_{02}+\sqrt{\Delta_{02}}}{2},\tfrac{x_{01}-\sqrt{\Delta_{01}}}{2},\tfrac{x_{02}-\sqrt{\Delta_{02}}}{2}\right),
  \end{equation*}
  where we choose $\sqrt{\Delta_{02}} = \sqrt{x_{02}^2-4x_2}$ in $M_v$,
  and $\sqrt{\Delta_{01}} = \sqrt{x_{01}^2-4x_1}$ such that
  $\sqrt{\Delta_{01}}\sqrt{\Delta_{02}} =
  x_{01}x_{02}-2x_{12}$. Locally in the $v$-adic topology, both maps
  are $C^1$-maps.
\end{lemma}

\begin{proof}
  Since $\Delta_{0i} \ne 0$ on $U_0$, Lemma~\ref{lem:openness} implies
  that $U_0^{(1)}$ is indeed open. Since $M_v$ is open, the same holds
  for $Z_0^{(1)}$. By construction of $M_v$ and because of
  \eqref{eq:Delta_square}, $\sqrt{\Delta_{02}}$ and
  $\sqrt{\Delta_{01}}$ are well-defined. To check that these maps are
  inverse to each other on the given sets, we argue as in the proof of
  Lemma~\ref{lem:X(k)_components}\eqref{it:Sym2_disjoint_union}: We
  note that $x_i=y_iz_i$ and $x_{0i}=y_i+z_i$ imply the equality
  $(X+y_i)(X+z_i)=X^2+x_iX+x_{0i}$ of polynomials; hence $y_i$ and
  $z_i$ must be $(x_{0i}\pm\sqrt{\Delta_{0i}})/2$ for an appropriate
  choice of square roots and signs, which is determined by considering
  $x_{12}=y_1z_2+y_2z_1$. The conditions in $U_0$ and $Z_0^{(1)}$ also
  match because of
  \begin{equation*}
    \Delta_{0i}=x_{0i}^2-4x_0x_i = (y_i+z_i)^2-4y_iz_i = (y_i-z_i)^2.
  \end{equation*}
  Since the choices of $y_2-z_2$ and of $\sqrt{\Delta_{02}}$ depend
  only on the first nontrivial digit in the $p$-adic case by
  definition of $M_v$, these constructions lead to maps that are
  locally $C^1$.
\end{proof}

Next, we consider $\epsilon \ne 1$ in $R_{\QQ_v}$.

\begin{lemma}\label{lem:C1-parameterization-epsilon}
  Let $v \in \Omega_\QQ$ and $\epsilon \ne 1$ in $R_{\QQ_v}$. The sets
  \begin{equation*}
    U_0^{(\epsilon)}:=U_0(\QQ_v) \cap X(\QQ_v)^{(\epsilon)},\
    Z_0^{(\epsilon)}:=\{(a_1,b_1,a_2,b_2) \in \QQ_v^{\oplus 4} : b_1 \ne 0,\ 2b_2 \in M_v\}
  \end{equation*}
  are open subsets of $X(\QQ_v)$, $\QQ_v^{\oplus 4}$, respectively, in
  the $v$-adic topology. The map
  \begin{equation*}
    (a_1,b_1,a_2,b_2) \mapsto (1:a_1^2-\epsilon b_1^2:a_2^2-\epsilon b_2^2:2a_1:2a_2:2(a_1a_2-\epsilon b_1b_2))
  \end{equation*}
  defines a bijection $Z_0^{(\epsilon)} \to U_0^{(\epsilon)}$ with inverse
  \begin{equation*}
    (1:x_1:x_2:x_{01}:x_{02}:x_{12}) \mapsto \left(\tfrac{x_{01}}{2},\tfrac{\sqrt{\Delta_{01}\epsilon^{-1}}}{2},\tfrac{x_{02}}{2},\tfrac{\sqrt{\Delta_{02}\epsilon^{-1}}}{2}\right),
  \end{equation*}
  where we choose $\sqrt{\Delta_{02}\epsilon^{-1}}$ in $M_v$, and
  $\sqrt{\Delta_{01}\epsilon^{-1}}$ such that
  $\epsilon\sqrt{\Delta_{01}\epsilon^{-1}}\sqrt{\Delta_{02}\epsilon^{-1}}
  = x_{01}x_{02}-2x_{12}$.  Locally in the $v$-adic topology, both
  maps are $C^1$-maps.
\end{lemma}

\begin{proof}
  We understand $\QQ_v(\sqrt{\epsilon})$ as
  $\QQ_v[X]/(X^2-\epsilon)$, where
  $\sqrt\epsilon$ is the class of
  $X$. We identify $\QQ_v(\sqrt{\epsilon})$ with $\QQ_v^{\oplus
    2}$ via the basis $1,\sqrt\epsilon$.  With
  $z_i=\overline{y_i}$ for $y_i=a_i+b_i\sqrt\epsilon \in
  \QQ_v(\sqrt\epsilon) \cong \QQ_v^{\oplus
    2}$, this is analogous to Lemma~\ref{lem:C1-parameterization-1}.
  Note that $Z_0^{(\epsilon)}$ corresponds to
  \begin{equation*}
    \{(y_1,y_2) \in \QQ_v(\sqrt\epsilon)^{\oplus 2} : y_1 \ne \overline{y_1},\ y_2-\overline{y_2}\in M_v\},
  \end{equation*}
  and the map corresponds to
  \begin{equation*}
    (y_1,y_2) \mapsto (1:y_1\overline{y_1}:y_2\overline{y_2}:y_1+\overline{y_1}:y_2+\overline{y_2}:y_1\overline{y_2}+y_2\overline{y_1}),
  \end{equation*}
  with inverse
  \begin{equation*}
    (1:x_1:x_2:x_{01}:x_{02}:x_{12}) \mapsto \left(\tfrac{x_{01}+\sqrt{\Delta_{01}}}{2},\tfrac{x_{02}+\sqrt{\Delta_{02}}}{2}\right),
  \end{equation*}
  where we choose $\sqrt{\Delta_{02}} = \sqrt{x_{02}^2-4x_2} \in M_v\sqrt{\epsilon}$
  and $\sqrt{\Delta_{01}} = \sqrt{x_{01}^2-4x_1}$ such that
  $\sqrt{\Delta_{01}}\sqrt{\Delta_{02}} = x_{01}x_{02}-2x_{12}$.
\end{proof}

\section{The asymptotic formula}\label{sec:asymptotic}

In this section, we prove Theorem~\ref{thm:main}. As discussed before
Lemma~\ref{lem:X(k)_components}, we can write $\Sym 2 2 \QQ(\QQ)$ as
the disjoint union $S_0 \cup S_+ \cup S_-$, where
\begin{itemize}
\item $S_0$ consists of unordered pairs $[\yy,\zz]$ of rational points
  $\yy,\zz \in \PP^2_\QQ(\QQ)$,
\item $S_+$ consists of unordered pairs of conjugate points
  $[\xx, \overline{\xx}]$ with $\xx \in \PP^2_\QQ(\QQ(\xx))$, where
  $\QQ(\xx)$ is a real quadratic field, and
\item $S_-$ consists of unordered pairs of conjugate points
  $[\xx, \overline{\xx}]$ with $\xx \in \PP^2_\QQ(\QQ(\xx))$, where
  $\QQ(\xx)$ is an imaginary quadratic field.
\end{itemize}
Let $\epsilon \in \{0,+,-\}$. By
Lemma~\ref{lem:X(k)_components}\eqref{it:Sym2_disjoint_union} and
Remark~\ref{rem:compare_notation_Q}, $X_\epsilon \subset X(\QQ)$ as in
\eqref{eq:def_X_0+-} is the image of the points
$[\yy,\zz]=[(y_0:y_1:y_2),(z_0:z_1:z_2)] \in S_\epsilon$ under the map
\begin{equation*}
  [\yy,\zz] \mapsto P_{[\yy,\zz]} := (y_0z_0:y_1z_1:y_2z_2:y_0z_1+y_1z_0:y_0z_2+y_2z_0:y_1z_2+y_2z_1)
\end{equation*}
induced by our isomorphism $\Symm^2(\PP^2_\QQ) \to X$ from
Proposition~\ref{prop:X_isom_Sym_2_2}.

The following result proves Theorem~\ref{thm:main} and gives the
values of the leading constants:

\begin{prop}\label{prop:asymptotic_formula}
  For $\epsilon \in \{0,+,-\}$, we have
  \begin{equation*}
    N(X_\epsilon,H,B) = c_\epsilon B \log B +
    \begin{cases}
      O(B), & \epsilon = 0,\\
      O(B(\log B)^{1/2}), & \epsilon \in \{+,-\},
    \end{cases}
  \end{equation*}
  with
  \begin{equation*}
    c_0 =\frac{c_{18}(2)}{24\zeta(3)^2}, \qquad c_+ = \frac{1}{9\zeta(3)^2}\cdot \frac{3\vol(S_1^+(e))}{4},\qquad c_- =
    \frac{1}{9\zeta(3)^2}\cdot \frac{6\vol(S_1^-)}{\pi},
  \end{equation*}
  where
  \begin{align*}
    S_1^+(e)&=\left\{(y_0,y_1,y_2,z_0,z_1,z_2) \in \RR^6 :
              \begin{aligned}
                &|y_0z_0|, |y_1z_1|, |y_2z_2|, |y_0z_1+y_1z_0| \le 1\\
                &|y_0z_2+y_2z_0|, |y_1z_2+y_2z_1| \le 1,\\
                &e^{-1} < \frac{\max\{|y_0|,|y_1|,|y_2|\}}{\max\{|z_0|,|z_1|,|z_2\}} \le
                1
              \end{aligned}
                  \right\},\\
    S_1^- &= \left\{(\xi_0+i\eta_0, \xi_1+i\eta_1, \xi_2+i\eta_2)\in \CC^{\oplus 3} :
            \begin{aligned}
              &\xi_j^2+\eta_j^2 \le 1\ (\forall j),\\
              &2|\xi_j\xi_k+\eta_j\eta_k|\le 1\ (\forall j \ne k)
            \end{aligned}
                \right\},
  \end{align*}
  and $c_{18}(2)$ is the leading constant of the volume of the (nonhomogeneously
  expanding) set
  \begin{equation*}
    T(X) = \left\{
      \begin{aligned}
        &(y_0,y_1,y_2,z_0,z_1,z_2) \in \RR^6 : |y_0z_0|, \dots,
        |y_1z_2+y_2z_1| \le X,\\ &\max\{|y_0|,|y_1|,|y_2|\} \ge 1,
        \max\{|z_0|,|z_1|,|z_2\} \ge 1
      \end{aligned}
    \right\},
  \end{equation*}
  (with $\vol(T(X)) = c_{18}(2) X^3\log X+O(X^3)$; see \cite[Theorem~4a,
  Lemma~13]{MR1330740}).
\end{prop}

\begin{proof}
  As explained around \cite[(1.9)]{MR1330740}, $\Sym 2 2 \QQ(\QQ)$ is
  in bijection to the set of nonzero quadratic forms
  $f \in \QQ[X_0,X_1,X_2]$ that are the product of two linear forms
  with algebraic coefficients. For each such $f$, let $H(f)$ be the
  exponential Weil height of the coefficient vector of $f$ considered
  as an element of $\PP^5(\QQ)$, and let $\QQ(f)$ be the field of
  definition of the linear forms. The map is given by
  \begin{equation*}
    [\yy,\zz]\mapsto f_{[\yy,\zz]}:=(y_0X_0+y_1X_1+y_2X_2)\cdot(z_0X_0+z_1X_2+z_2X_2).
  \end{equation*}
  Then $Z_8(2,X)$ is defined as the number of such decomposable quadratic
  forms with $H(f) \le X$, where proportional forms are counted as one. Let
  $Z_8^+(2,X)$, $Z_8^-(2,X)$, $Z_8^0(2,X)$ be the number of such forms $f$
  with $\QQ(f)$ real quadratic, imaginary quadratic, rational,
  respectively. For $\epsilon \in \{0,+,-\}$, we observe that the forms counted by
  $Z_8^\epsilon(2,X)$ correspond to $S_\epsilon$.
  Clearly $H(f_{[\yy,\zz]})^3 = H(P_{[\yy,\zz]})$. Therefore,
  $N(X_\epsilon, H, B) = Z_8^\epsilon(2,B^{1/3})$. Hence the result follows
  from \cite[Theorem~4a]{MR1330740}.
\end{proof}

\section{The expected asymptotic formula}\label{sec:expectation}

We identify $\Sym 2 2 \QQ$ with our singular cubic fourfold $X$
defined by (\ref{eq:cubic}) using the isomorphism described in
Proposition~\ref{prop:X_isom_Sym_2_2}. To interpret the asymptotic
formula in Proposition~\ref{prop:asymptotic_formula}, we must consider
its desingularization, which is the Hilbert scheme $\Hilb 2 2 \QQ$:

Hilbert schemes were constructed by Grothendieck \cite{MR1611822}.
For the Hilbert scheme $\Hilb 2 2 \QQ$ of subschemes of length $2$ on
$\PP^2_\QQ$, see \cite{MR0237496,MR0335512} and also
\cite{MR3010070}. The \emph{Hilbert--Chow morphism}
\begin{equation*}
  h: \Hilb 2 2 \QQ \to \Sym 2 2 \QQ
\end{equation*}
is the blow-up of the singular locus of $\Sym 2 2 \QQ$ described in
Remark~(\ref{rem:singular_locus}); hence we can identify it with
$h : \tX \to X$ as in Theorem~\ref{thm:manin-peyre}.

\begin{lemma}\label{lem:almost_Fano}
  The desingularization $\tX \cong \Hilb 2 2 \QQ$ of
  $X \cong \Sym 2 2 \QQ$ is a split rational almost Fano variety
  \cite[D\'efinition~3.1]{MR2019019} of Picard number $\rho_\tX = 2$
  whose rational points are not thin.
\end{lemma}

\begin{proof}
  Since we can resolve the defining equation (\ref{eq:cubic}) for the
  coordinate $X_0$, for example, our cubic fourfold $X$ is rational, and hence its
  rational points are not a thin set. Since $h$ is birational, the
  same holds for $\tX$. Furthermore, $\omega_\tX^{-1}$ is big and nef
  (see the proof of \cite[Theorem~2.5]{MR3010070}), and the geometric
  Picard group of $\tX$ is free of rank $\rho_\tX = 2$ \cite[\S
  1]{MR0335512}. Since $\Gal(\Qbar/\QQ)$ fixes both the exceptional
  divisor and $h^*\OO_X(1)$, we observe that $\tX$ is split (i.e., the
  natural map $\Pic \tX \to \Pic(\tX_\Qbar)$ is an isomorphism).
\end{proof}

\begin{lemma}\label{lem:anticanonical_height}
  The functions $H$ as in \eqref{eq:height} and $\tH = H \circ h$
  define anticanonical heights on $X\cong \Sym 2 2 \QQ$ and
  $\tX\cong \Hilb 2 2 \QQ$, respectively.
\end{lemma}

\begin{proof}
  Indeed, $\Pic(\Sym 2 2 \QQ) \cong \ZZ$; let $\OO_{\Sym 2 2 \QQ}(1)$
  be its ample generator. This corresponds to $\OO_X(1)$ via our
  isomorphism $\Sym 2 2 \QQ \cong X \subset \PP^5_\QQ$. The
  anticanonical class of $X \subset \PP^5_\QQ$ is $\OO_X(a)$ where
  $a=(n+1)-d=3$ for this hypersurface of degree $d=3$ in
  $n=5$-dimensional projective space. Hence the third power of the
  exponential Weil height on $\PP^5_\QQ$ restricts to an anticanonical
  height on $X$.

  By \cite{MR0237496} (see \cite[Theorem~2.2]{MR3010070}), $h$ is a
  crepant resolution (i.e., $h^*\omega_X^{-1} = \omega_\tX^{-1}$).  We
  will see in Lemma~\ref{lem:height_via_adelic_norm} that $\tH$ is
  induced by an adelic norm on $\omega_\tX^{-1}$.
\end{proof}

Therefore, the version \cite[Conjecture~5.2]{MR4472281} of Manin's
conjecture should apply to $\tX$ with the anticanonical height $\tH$.

We observe that $\tX_0:=h^{-1}(X_0) \subset \tX(\QQ)$ is a thin set
since $S_0 \subset \Sym 2 2 \QQ(\QQ)$ is thin (as discussed in
\cite{MR4057715} and \cite[\S 8.4.2]{MR3741845}) and corresponds to
$X_0 \subset X(\QQ)$.  By \cite[Example~8.6]{MR3712166},
\cite[Example~5.14]{MR4472281}, $\tX_0$ is the expected accumulating
set that should be removed from the counting function according to
\cite[Conjecture~5.2]{MR4472281}.

Hence we must count points on
\begin{equation*}
  \tX(\QQ) \setminus \tX_0 = \tX_+ \cup \tX_-,
\end{equation*}
where $\tX_+ = h^{-1}(X_+)$ is the subset of $\Hilb 2 2 \QQ(\QQ)$
corresponding to subschemes of $\PP^2_\QQ$ of length $2$ consisting of
unordered pairs of conjugate points defined over real quadratic
extensions of $\QQ$, and $\tX_- = h^{-1}(X_-)$ corresponds to those
over imaginary quadratic extensions.

\begin{prop}\label{prop:asymptotic_tX}
  We have
  \begin{align*}
    N(\tX(\QQ) \setminus \tX_0,\tH,B)
    &:=\#\{\xx \in \tX(\QQ) \setminus \tX_0 : \tH(\xx)\le B\}\\
    &= (c_++c_-) B \log B + O(B(\log B)^{1/2}).
  \end{align*}
\end{prop}

\begin{proof}
  Since $h$ is an isomorphism outside the singular locus of $X$
  (whose rational points $X(\QQ)_\Delta$ are in $X_0$ by
  Lemma~\ref{lem:X(k)_components}\eqref{it:X(k)_disjoint_union} and
  Remark~\ref{rem:compare_notation_Q}), this number is equal to
  $N(X_+\cup X_-,H,B)$. Hence the result follows directly from
  Proposition~\ref{prop:asymptotic_formula}
\end{proof}

\begin{remark}\label{rem:exceptional_points_tX}
  The number of rational points of height $\le B$ in the singular
  locus of $X$ is $O(B^{1/2})$ (since (\ref{eq:diagonal_map}) sends a
  point of exponential Weil height $B$ in $\PP^2(\QQ)$ to a point of
  height $\asymp B^6$ on $X$). Therefore, the number of rational
  points of bounded height in $\tX_0$ outside the exceptional divisor
  is
  \begin{equation*}
    N(\tX_0 \setminus E(\QQ),\tH,B) = N(X_0,\tH,B)+O(B^{1/2}) = c_0 B \log B + O(B).
  \end{equation*}
  Hence the thin set $\tX_0 \setminus E(\QQ)$ has the same growth rate
  $\asymp B \log B$ for rational points of bounded height as
  $\tX_+ \cup \tX_-$. Note that the exceptional divisor $E$ contains
  infinitely many rational points of bounded height.
\end{remark}

We claim that the asymptotic formula in
Proposition~\ref{prop:asymptotic_tX} agrees with
\cite[Conjecture~5.2]{MR4472281}. First of all, the exponents of $B$
and $\log B$ are as expected: Since $\tX$ is an anticanonical height,
the expected exponent of $B$ is $a(\tX,\omega_\tX^{-1}) = 1$. Since
$\tX$ is split (see Lemma~\ref{lem:almost_Fano}), its Picard group
over $\QQ$ coincides with its geometric Picard group, hence the
expected exponent of $\log B$ is
$b(\QQ,\tX,\omega_\tX^{-1}) - 1 = \rho_\tX -1=1$.  Since $\tX$ is an
almost Fano variety (see
Lemma~\ref{lem:almost_Fano}), we should compare the leading constant
$c_++c_-$ to Peyre's constant for $\tX$ as in \cite[Formule empirique
5.1]{MR2019019}, namely
\begin{equation}\label{eq:peyres_constant}
  \theta_\tH(\tX) = \alpha(\tX)\beta(\tX)\tau_\tH(\tX).
\end{equation}

Here, by \cite[Proposition~3.2]{MR4057715},
\begin{equation}\label{eq:alpha}
  \alpha(\tX) = \frac 1 9
\end{equation}
(which does not depend on the chosen adelic norm). Since
$\tX \cong \Hilb 2 2 \QQ$ is split (Lemma~\ref{lem:almost_Fano}),
$\Gal(\Qbar/\QQ)$ acts trivially on the geometric Picard group
$\Pic(\tX_\Qbar)$, hence
\begin{equation}\label{eq:beta}
  \beta(\tX) = \#H^1(\Gal(\Qbar/\QQ), \Pic(\tX_\Qbar)) = 1.
\end{equation}

To compute $\tau_\tH(\tX)$, we must construct an adelic norm
$(\|\cdot\|_v)_{v \in \Omega_\QQ}$ on the anticanonical bundle
$\omega_\tX^{-1}$ that induces our height function $\tH$.

We argue as in \cite[\S 13]{MR3552013}. Let
$\mathcal{C}_{X/\PP^5_\QQ}$ be the conormal sheaf of our singular
cubic fourfold $X \subset \PP^5$, and $\Omega_{\PP^5_\QQ}^5$ the top
exterior power of the sheaf of differentials on $\PP^5_\QQ$. Then the
canonical sheaf of $X$ is
\begin{equation*}
  \omega_X = \det(\mathcal{C}_{X/\PP^5_\QQ})^\vee \otimes_{\OO_X} (\Omega_{\PP^5_\QQ}^5)_{|X}.
\end{equation*}
Next, we construct an explicit isomorphism $\omega_X \cong \OO_X(-3)$.

Since $X$ is a cubic hypersurface in $\PP^5_\QQ$, multiplication by
the defining polynomial $P$ induces an isomorphism
$\OO_X(-3) \to \mathcal{C}_{X/\PP^5_\QQ}$. Hence we obtain an
isomorphism $\OO_X(3) \to \det(\mathcal{C}_{X/\PP^5_\QQ})^\vee$ that
is defined on the affine subset $\AAA^5_\QQ$ where $X_\lambda \ne 0$
(for any $\lambda$ in the index set
$\Lambda = \{0,1,2,(01),(02),(12)\}$) by
$X_\lambda^3 \mapsto (P^{(\lambda)})^\vee$ (where
$P^{(\lambda)}=P/X_\lambda^3$ is the affine polynomial in
$X_\mu^{(\lambda)} = X_\mu/X_\lambda$ defining $X$ on the affine chart
where $X_\lambda \ne 0$). The map
\begin{equation}\label{eq:omega_S=O(-3)}
  X_\lambda^{-6} \mapsto \pm \dd X_0^{(\lambda)} \wedge \dots \wedge \widehat{\dd X_\lambda^{(\lambda)}} \wedge \dots \wedge \dd X_{12}^{(\lambda)}
\end{equation}
(with a suitable choice of signs that will not matter for our
purposes) defines an isomorphism
$\OO_X(-6) \to (\Omega_{\PP^5_\QQ}^5)_{|X}$. Together, we have an
isomorphism $\OO_X(-3) \to \omega_S$ defined on $X^{(\lambda)}$ by
\begin{equation*}
  X_\lambda^{-3} \mapsto s_\lambda = \pm (P^{(\lambda)})^\vee \otimes \dd X_0^{(\lambda)} \wedge \dots \wedge \widehat{\dd X_\lambda^{(\lambda)}} \wedge \dots \wedge \dd X_{12}^{(\lambda)}.
\end{equation*}
On $X_\reg$ (the complement of the singular locus of $X$), the conormal
sequence
\begin{equation*}
  0 \to \mathcal{C}_{X_\reg/\PP^5_\QQ} \to (\Omega_{\PP^5_\QQ}^1)_{|X} \to  \Omega_{X_\reg}^1 \to 0
\end{equation*}
allows us to identify $\omega_{X_\reg}$ with $\Omega^2_{X_\reg}$, with
\begin{equation}\label{eq:s_lambda}
  s_\lambda = \pm (\partial P^{(\lambda)}/\partial X_\mu^{(\lambda)})^{-1} \dd X_0^{(\lambda)} \wedge \dots \wedge \widehat{\dd X_\lambda^{(\lambda)}} \wedge \dots \wedge \widehat{\dd X_\mu^{(\lambda)}} \wedge \dots \wedge \dd X_{12}^{(\lambda)}
\end{equation}
(for any index $\mu \ne \lambda$ in $\Lambda$) in
$\omega_{X,\xi} = \omega_{X_\reg,\xi} = \Omega^2_{\QQ(X)}$ for the
generic point $\xi$ on $X$. Dual to the rational sections $s_\lambda$
of $\omega_X$, we obtain global sections $\tau_\lambda$ of
$\omega_X^{-1}$ corresponding to the sections $X_\lambda^3$ of
$\OO_X(3)$.

The pullback of differential forms along the crepant resolution
$h : \tX \to X$ allows us to identify the canonical bundle
$\omega_\tX$ with $h^*\omega_X$, and analogously for the anticanonical
bundles. In particular, $h^*\tau_\lambda$ are global sections of
$\omega_\tX^{-1}$, and they form a basepoint free linear system since
$(h^*X_0,\dots,h^*X_{12})$ define the morphism
$\iota\circ h: \tX \to \PP^5_\QQ$ (where $\iota$ is the
embedding of $X \hookrightarrow \PP^5_\QQ$ defined by \eqref{eq:cubic}). As in
\cite[Exemples 7.7.5 b)]{peyre-book}, we use them to define the
$v$-adic norm
\begin{equation*}
  \|\tau\|_v := \min_{\substack{\mu \in \Lambda\\h^*\tau_\mu(x) \ne 0}} \left\{\left|\frac{\tau}{h^*\tau_\mu(x)}\right|_v\right\}
\end{equation*}
for $\tau \in \omega_\tX^{-1}(x)$ over a local point
$x \in \tX(\QQ_v)$, and by \cite[Exemples 7.7.16~b)]{peyre-book},
$(\|\cdot\|)_{v \in \Omega_\QQ}$ is an adelic norm on the
anticanonical bundle of $\tX$.

\begin{lemma}\label{lem:height_via_adelic_norm}
  The height induced (as in \cite[D\'efinition~2.3]{MR2019019}) by
  this adelic norm coincides with our height function
  $\tH = H \circ h$.
\end{lemma}

\begin{proof}
  This holds by \cite[Exemples 7.7.16 b)]{peyre-book}, but can also be
  checked directly: For a rational point $x \in \tX(\QQ)$, its image
  $h(x) = (x_0:\dots:x_{12}) \in X(\QQ)$ is represented by coprime
  integral coordinates. At least one of them, say $x_\lambda$, does not
  vanish, i.e., $h^*\tau_\lambda(x) \ne 0$. Then the corresponding
  height is
  \begin{align*}
    &\prod_{v \in \Omega_\QQ} \|h^*\tau_\lambda(x)\|_v^{-1} = \prod_{v \in \Omega_\QQ} \max_{\mu \in \Lambda}\{|X^{(\lambda)}_\mu(h(x))|_v\}^3= \prod_{v \in \Omega_\QQ} \frac{\max\{|x_0|_v,\dots,|x_{12}|_v\}^3}{|x_\lambda|_v^3} \\
    &= \max\{|x_0|_v,\dots,|x_{12}|_v\}^3 = H(h(x)) = \tH(x)
  \end{align*}
  by
  $\tau_\mu/\tau_\lambda = (X_\mu/X_\lambda)^3 = (X^{(\lambda)}_\mu)^3$
  from \eqref{eq:omega_S=O(-3)}, and by the product formula.
\end{proof}

Our adelic norm also defines a Tamagawa measure, which we now make
explicit. The Hilbert--Chow morphism $h$ is an isomorphism between the
complement $\tX \setminus E$ of the exceptional divisor $E$ in $\tX$
and $X_\reg$ (the complement of the diagonal, which is the radical of
the vanishing locus of $\Delta_{01},\Delta_{02},\Delta_{12}$ by
Remark~\ref{rem:singular_locus}). On the affine subset $X^{(0)}$ of
$X$ where $X_0 \ne 0$, we treat $X_1^{(0)}$ as the dependent
coordinate, and then the four remaining $X_\nu^{(0)}$ (for
$\nu \in \Lambda \setminus \{0,1\}$) form a system of local
coordinates on $U_0\subset X^{(0)} \subset X$ as in \eqref{eq:U_0},
defining a homeomorphism
\begin{equation*}
  \psi: U_0(\QQ_v) \to W_0 := \{\xx = (x_2,x_{01},x_{02},x_{12}) \in \QQ_v^{\oplus 4} \mid x_{01}^2-4x_1 \ne 0,\ x_{02}^2-4x_2 \ne 0\}.
\end{equation*}
Therefore, with $V_0:=h^{-1}(U_0)$, the four $Y_\nu:=h^*X_\nu^{(0)}$
define local coordinates on $V_0(\QQ_v) \subset \tX(\QQ_v)$ via the
map $\phi= \psi \circ h : V_0 \to W_0$.

With the notation $\xx = (x_2,x_{01},x_{02},x_{12}) \in \QQ_v^{\oplus 4}$, we have
\begin{equation}\label{eq:local_integral}
  \begin{aligned}
    &\omega_{\tH,v}(V_0(\QQ_v))
      = \int_{W_0} \left\|\left(\frac{\partial}{\partial Y_2}\wedge \frac{\partial}{\partial Y_{01}}\wedge \frac{\partial}{\partial Y_{02}}\wedge \frac{\partial}{\partial Y_{12}}\right)(\phi^{-1}(\xx))\right\|_v \dd \xx\\
    &= \int_{W_0} \min_{\substack{\lambda \in \Lambda\\\tau_\lambda(\psi^{-1}(\xx)) \ne 0}} \left\{\left|\frac{\tau_0(\psi^{-1}(\xx))}{(\partial P^{(0)}/\partial X_1^{(0)}\cdot \tau_\lambda)(\psi^{-1}(\xx))}\right|_v\right\} \dd \xx\\
    &= \int_{W_0} \frac{\dd \xx}{|x_{02}^2-4x_2|_v \cdot  \max\{1,|\frac{x_{01}x_{02}x_{12}-x_{12}^2-x_2x_{01}^2}{x_{02}^2-4x_2}|_v,|x_2|_v,|x_{01}|_v,|x_{02}|_v,|x_{12}|_v\}^3}.
  \end{aligned}
\end{equation}
Here, we have applied $\tau_\lambda / \tau_0 = (X_\lambda/X_0)^3 = (X_\lambda^{(0)})^3$,
and
\begin{equation*}
  \frac{\partial}{\partial X_2^{(0)}}\wedge \frac{\partial}{\partial
    X_{01}^{(0)}}\wedge \frac{\partial}{\partial X_{02}^{(0)}}\wedge
  \frac{\partial}{\partial X_{12}^{(0)}} = \pm \left(\frac{\partial P^{(0)}}{\partial
  X_1^{(0)}}\right)^{-1} \tau_0 = \pm ((X_{02}^{(0)})^2-4X_2^{(0)})^{-1} \tau_0
\end{equation*}
(using \eqref{eq:s_lambda} for $(\lambda,\mu)=(0,1)$).

The Tamagawa number as \cite[D\'efinition~4.8]{MR2019019} appearing in
\eqref{eq:peyres_constant} is
\begin{equation}\label{def:tamagawa}
  \tau_\tH(\tX) = \omega_\tH(\tX(\AAA_\QQ))
\end{equation}
since $\tX$ is rational and hence satisfies weak approximation. Here,
$\omega_\tH$ is the Tamagawa measure
\begin{equation*}
  \omega_\tH = \lim_{s \to 1}(s-1)^{\rho_\tX} L_S(s,\Pic(\tX_\Qbar)) \prod_{v \in \Omega_\QQ} \lambda_v^{-1} \omega_{\tH,v}
\end{equation*}
as in \cite[D\'efinition~4.6]{MR2019019}, where $S\subset \Omega_\QQ$ is a suitable
finite set of finite places of $\QQ$. Since $\tX$ is split of Picard
number $\rho_\tX=2$, we have
\begin{equation*}
  L_p(s,\Pic(\tX_\Qbar)) = \frac{1}{(1-\frac 1{p^s})^2}
\end{equation*}
and hence
\begin{equation}\label{eq:convergence_factors}
  \lambda_p^{-1} = L_p(1,\Pic(\tX_\Qbar))^{-1} = (1-\tfrac 1 p)^2
\end{equation}
for $p \in \Omega_\QQ \setminus (\{\infty\} \cup S)$, while
$\lambda_\infty^{-1} = \lambda_p^{-1} = 1$ for $p \in S$. Hence
\begin{equation*}
  L_S(s,\Pic(\tX_\Qbar)) = \prod_{p \in \Omega_\QQ \setminus (\{\infty\} \cup S)} L_p(s,\Pic(\tX_\Qbar)) = \prod_{p \in \Omega_\QQ \setminus (\{\infty\} \cup S)} \frac{1}{(1-\frac 1{p^s})^2}.
\end{equation*}
We observe that $\omega_\tH$ is unchanged if we replace $S$ by
$\emptyset$ since the extra factors in $L_S(s,\Pic(\tX_\Qbar))$ and
$\lambda_p^{-1}$ cancel out. Then the $L$-function is $\zeta(s)^2$,
so that
\begin{equation}\label{eq:limit}
  \lim_{s \to 1}(s-1)^2 L_S(s,\Pic(\tX_\Qbar))=1.
\end{equation}
It remains to compute the local densities $\omega_{\tH,v}(\tX(\QQ_v))$.

\begin{lemma}\label{lem:omega_v_parts}
  For any place $v \in \Omega_\QQ$ and $R_{\QQ_v}$ as in
  Definition~\ref{def:R_k}, we have
  \begin{equation*}
    \omega_{\tH,v}(\tX(\QQ_v)) = \sum_{\epsilon \in R_{\QQ_v}} \omega_v^{(\epsilon)},
  \end{equation*}
  where
  \begin{equation*}
    \omega_v^{(1)} = \frac 1 2 \int_{\QQ_v^{\oplus 4}} \frac{\dd y_1 \dd y_2 \dd z_1
      \dd z_2}{F_1(y_1,y_2,z_1,z_2)^3},\qquad 
    \omega_v^{(\epsilon)} = \frac{|4\epsilon|_v}{2} \int_{\QQ_v^{\oplus 4}} \frac{\dd a_1 \dd b_1 \dd a_2 \dd b_2}{F_\epsilon(a_1,b_1,a_2,b_2)^3}
  \end{equation*}
  for $\epsilon \ne 1$, with 
  \begin{align*}
    F_1(y_1,y_2,z_1,z_2) &= \max\{1,|y_1z_1|_v, |y_2z_2|_v, |z_1+y_1|_v, |z_2+y_2|_v, |y_1z_2+y_2z_1|_v\},\\
    F_\epsilon(a_1,b_1,a_2,b_2) &= \max\{1,|a_1^2-\epsilon b_1^2|_v,|a_2^2-\epsilon b_2^2|_v,|2a_1|_v,|2a_2|_v,|2(a_1a_2-\epsilon b_1b_2)|_v\}.
  \end{align*}
\end{lemma}

\begin{proof}
  Our starting point is (\ref{eq:local_integral}).
  Lemma~\ref{lem:X(k)_components} and Lemma~\ref{lem:openness} allow
  us to split the $v$-adic integral over an open subset of $X(\QQ_v)$
  into integrals over the corresponding open subsets of
  $X(\QQ_v)^{(\epsilon)}$.

  Lemma~\ref{lem:C1-parameterization-1} provides a parameterization of
  the dense open subset $U_0^{(1)} = U_0(\QQ_v) \cap X(\QQ_v)^{(1)}$
  of $X(\QQ_v)^{(1)}$ by the open subset $Z_0^{(1)}$ of $\QQ_v^{\oplus 4}$. This
  allows us to make the change of variables
  \begin{equation*}
    x_2=y_2z_2,\quad x_{01}=y_1+z_1,\quad x_{02}=y_2+z_2,\quad x_{12}=y_1z_2+y_2z_1,
  \end{equation*}
  with Jacobian
  \begin{equation*}
    \det
    \begin{pmatrix}
      0 & 1 & 0 & z_2\\
      z_2 & 0 & 1 & z_1\\
      0 & 1 & 0 & z_2\\
      y_2 & 0 & 1 & z_1
    \end{pmatrix} = -(y_2-z_2)^2,
  \end{equation*}
  which turns the expression for $x_1$ into $y_1z_1$ and $x_{0i}^2-4x_i$ into
  $(y_i-z_i)^2$. Hence
  \begin{equation*}
    \omega_{\tH,v}(h^{-1}(U_0^{(1)})) = \int_{Z_0^{(1)}} \frac{\dd y_1 \dd y_2 \dd z_1 \dd z_2}
    {\max\{1,|y_1z_1|,|y_2z_2|,|z_1+y_1|,|z_2+y_2|,|y_1z_2+y_2z_1|\}^3}.
  \end{equation*}
  By symmetry and since $y_i=z_i$ defines a lower-dimensional set
  where the integral is $0$, we can replace $Z_0^{(1)}$ by $\QQ_v^{\oplus 4}$ if we
  introduce a factor $1/2$.
  
  Similarly, Lemma~\ref{lem:C1-parameterization-epsilon} gives a
  parameterization of $U_0^{(\epsilon)}$ by the open subset
  $Z_0^{(\epsilon)}$ of $\QQ_v^{\oplus 4}$. The change of variables
  \begin{equation*}
    x_2 = a_2^2-\epsilon b_2^2,\quad
    x_{01}=2a_1,\quad x_{02}=2a_2,\quad 
    x_{12} = 2(a_1a_2-\epsilon b_1b_2)
  \end{equation*}
  has Jacobian
  \begin{equation*}
    \det
    \begin{pmatrix}
      0 & 2 & 0 & 2a_2\\
      2a_2 & 0 & 2 & 2a_1\\
      0 & 0 & 0 & -2\epsilon b_2\\
      -2\epsilon b_2 & 0 & 0 & -2\epsilon b_1
    \end{pmatrix} = -16\epsilon^2 b_2^2
  \end{equation*}
  and turns the expression for $x_1$ into $a_1^2-\epsilon b_1^2$, and
  $x_{02}^2-4x_2$ into $4\epsilon b_2^2$.  Therefore,
  \begin{equation*}
    \omega_{\tH,v}(h^{-1}(U_0^{(\epsilon)}))\! =\!\! \int_{Z_0^{(\epsilon)}} \frac{|4\epsilon|_v \dd a_1 \dd a_2 \dd b_1 \dd b_2}
    {\max\{1,|a_1^2-\epsilon b_1^2|,|a_2^2-\epsilon b_2^2|,|2a_1|,|2a_2|,|2(a_1a_2-\epsilon b_1b_2)|\}^3}.
  \end{equation*}
  Again by symmetry, we may remove the condition $b_1 \ne 0$ and
  $2b_2 \in M_v$ in the definition of $Z_0^{(\epsilon)}$ while
  introducing a factor $1/2$ to arrive at $\omega_v^{(\epsilon)}$.

  We note that $U_0(\QQ_v)$ is the disjoint union of its subsets
  $U_0^{(\epsilon)}$ for $\epsilon \in R_{\QQ_v}$ by
  Lemma~\ref{lem:X(k)_components}\eqref{it:X(k)_disjoint_union} (where
  we avoid $X(\QQ_v)_\Delta$ because of the condition
  $\Delta_{02} \ne 0$). Furthermore, the complement of $U_0(\QQ_v)$ in
  $X(\QQ_v)$ has lower dimension, and hence does not contribute to the
  integral. Therefore, all $\omega_v^{(\epsilon)}$ add up to
  $\omega_{\tH,v}(\tX(\QQ_v))$.
\end{proof}

\begin{lemma}\label{lem:compute_F_epsilon}
  Let $p$ be a prime. For $y_1,y_2,z_1,z_2 \in \QQ_p$, let
  $\alpha=\min\{v_p(y_1),v_p(y_2)\}$,
  $\beta=\min\{v_p(z_1),v_p(z_2)\}$. Then
  \begin{equation*}
    F_1(y_1,y_2,z_1,z_2)) =
    \begin{cases}
      1, & \alpha \ge 0,\ \beta \ge 0,\\
      p^{-\alpha}, & \alpha < 0,\ \beta \ge 0,\\
      p^{-\beta}, & \alpha \ge 0,\ \beta < 0,\\
      p^{-(\alpha+\beta)}, & \alpha < 0, \beta < 0.
    \end{cases}
  \end{equation*}

  Consider $\epsilon \ne 1$ in $R_{\QQ_p}$. For
  $a_1,b_1,a_2,b_2 \in \QQ_p$, let
  $\alpha = \min\{v_p(a_1),v_p(a_2)\}$ and
  $\beta = \min\{v_p(b_1),v_p(b_2)\}$. Then
  \begin{equation*}
    F_\epsilon(a_1,b_1,a_2,b_2) =
    \begin{cases}
      1, & 0 \le \alpha,\ 0 \le \beta,\\
      p^{-2\alpha}, & \alpha < 0,\ \alpha < \beta,\\
      p^{-2\alpha-\delta}, & \alpha < 0,\ \alpha = \beta,\\
      p^{-2\beta-v_p(\epsilon)}, & \beta < 0,\ \beta < \alpha,
    \end{cases}
  \end{equation*}
  Here, if $p=2$ and either
  \begin{itemize}
  \item $v_p(a_1)=v_p(b_1)<\min\{v_p(a_2),v_p(b_2)\}$, or
  \item $v_p(a_2)=v_p(b_2)<\min\{v_p(a_1),v_p(b_1)\}$, or
  \item $v_p(a_1)=v_p(a_2)=v_p(b_1)=v_p(b_2)$,
  \end{itemize}
  we have $\delta=1$ for $\epsilon \in \{3,7\}$ and $\delta=2$ for
  $\epsilon=5$. In all other cases, we have $\delta=0$.
\end{lemma}

\begin{proof}
  For $\epsilon \ne 1$ and $p\ne 2$, we observe
  $|a_i^2 -\epsilon b_i^2|_p = \max\{|a_i^2|_p,|\epsilon b_i^2|_p\}$ since
  $v_p(u-\epsilon v)=0$ for $u,v\in \ZZ_p^\times$ because
  $\epsilon \notin \QQ_p^{\times 2}$. Hence
  $|2(a_1a_2-\epsilon b_1b_2)|_p \le \max\{|a_1a_2|_p,|\epsilon
  b_1b_2|_p\}$ cannot be larger than both
  $|a_1^2 -\epsilon b_1^2|,|a_2^2 -\epsilon b_2^2|$. Furthermore,
  $|2a_i|_v$ cannot be larger than both $1$ and $|a_i^2|_p$.

  For $p=2$, $\epsilon \ne 1$, $|\epsilon|_2=1$ and $a_i=2^\alpha u$,
  $b_i=2^\alpha v$ with $u,v \in \ZZ_2^\times$, we have
  $|a_i^2-\epsilon b_i^2|_2 = 2^{-2\alpha-\delta}$ since
  $u^2,v^2 \equiv 1 \pmod 8$, so that $v_2(u^2-\epsilon v^2)$ is $1$
  for $\epsilon = 3,7$ and $2$ for $\epsilon = 5$.  We observe that
  this affects $F_\epsilon(a_1,b_1,a_2,b_2)$ only in the three cases listed
  in our statement. (We note that also in these cases,
  $|2(a_1a_2-\epsilon b_1b_2)|_2 \le 2^{-2}\max\{|a_1a_2|_2,|\epsilon
  b_1b_2|_2\}$ cannot be the maximum.)
  
  The case $F_1$ is similar, but easier.
\end{proof}

\begin{lemma}\label{lem:p-adic_densities}
  For any prime $p$, we have
  \begin{equation*}
    \omega_p^{(1)} = \frac 1 2 \left(1+\frac 1 p+\frac 1{p^2}\right)^2,\qquad \omega_p^{(\epsilon)} = \frac{|4\epsilon|_p}{2}
    \begin{cases}
      1+\frac 1 p+\frac 1{p^2}, & v_p(\epsilon)=1,\\
      1+\frac 1{p^2}+\frac 1{p^4}, & p \ne 2,\ \epsilon = u,\\
      \frac{7}{4}, & p = 2,\ \epsilon = 3,7,\\
      \frac{21}{4}, & p = 2,\ \epsilon = 5,
    \end{cases}
  \end{equation*}
  and
  \begin{equation*}
    \omega_{\tH,p}(\tX(\QQ_p)) = \left(1+\frac 1 p+\frac 1{p^2}\right)^2.
  \end{equation*}
\end{lemma}

\begin{proof}
  Let $\epsilon \ne 1$. In the notation of
  Lemma~\ref{lem:compute_F_epsilon}, we first compute the contribution
  to $\omega_v^{(\epsilon)}$ as in Lemma~\ref{lem:omega_v_parts} of
  the four cases where $\{i,j\}=\{k,l\}=\{1,2\}$ and
  $\alpha=v_p(a_i)<v_p(a_j)$ and $\beta=v_p(b_k)<v_p(b_l)$. Since our
  $p$-adic integrals are computed with respect to the Haar measure
  $\mu_p$ on $\QQ_p$ normalized such that
  \begin{equation*}
    \mu_p\{x \in \QQ_p : v_p(x)=\gamma\} = \left(1-\frac 1 p\right) p^{-\gamma},\quad
    \mu_p\{x \in \QQ_p: v_p(x)>\gamma\} = \frac 1 p \cdot p^{-\gamma}
  \end{equation*}
  for any $\gamma \in \ZZ$, we obtain
  \begin{align*}
    &\sum_{\alpha,\beta \in \ZZ} \frac{|4\epsilon|_p \cdot \mu_p\{v_p(a_1)=\alpha\}\cdot \mu_p\{v_p(a_2)>\alpha\}  \cdot  \mu_p\{v_p(b_1)=\beta\}\cdot \mu_p\{v_p(b_2)>\beta\}}{2F_\epsilon(1,y_1,y_2)}\displaybreak\\
    &=\frac{|4\epsilon|_p}{2} \left(1-\frac 1 p\right)^2\left(\frac 1 p\right)^2\cdot\\
    &\left(\sum_{\alpha,\beta \ge 0} \frac{p^{-2\alpha-2\beta}}{1}+\sum_{\alpha<0,\alpha<\beta} \frac{p^{-2\alpha-2\beta}}{p^{-6\alpha}}+\sum_{\alpha=\beta<0} \frac{p^{-2\alpha-2\beta}}{p^{-6\alpha-3\delta}}+\sum_{\beta<0,\alpha>\beta} \frac{p^{-2\alpha-2\beta}}{p^{-6\beta-3v_p(\epsilon)}}\right)\\
    &=\frac{|4\epsilon|_p}{2}\left(1-\frac 1 p\right)^2\left(\frac 1 p\right)^2\left(\frac{1}{\left(1-\frac 1 {p^2}\right)^2}+\frac{p^{-4}}{\left(1-\frac 1 {p^2}\right)^2}+\frac{p^{3\delta-2}}{1-\frac 1 {p^2}}+\frac{p^{3v_p(\epsilon)-4}}{\left(1-\frac 1 {p^2}\right)^2}\right)\\
    &=\frac{|4\epsilon|_p}{2}\left(\frac 1 p\right)^2\frac{1+p^{-4}+p^{3\delta-2}-p^{3\delta-4}+p^{3v_p(\epsilon)-4}}{\left(1+\frac 1 p\right)^2}.
  \end{align*}
  Note that $\delta \ne 0$ in precisely two of the four cases when
  $p=2$ and $\epsilon \in \{3,5,7\}$. The four cases where either
  $v_p(a_i)<v_p(a_j)$ or $v_p(b_k)<v_p(b_l)$ is replaced by an
  equality give the same result, but with one factor $\frac 1 p$ replaced by
  $1-\frac 1 p$; in this case, we always have $\delta=0$. Finally, the
  case where $v_p(a_1)=v_p(a_2)$ and $v_p(b_1)=v_p(b_2)$ gives the same
  as the first cases (possibly with $\delta > 0$), but with the factor
  $(\frac 1 p)^2$ replaced by $(1-\frac 1 p)^2$.

  The case $\epsilon = 1$ is similar, but easier. To compute
  $\omega_{\tH,p}(\tX(\QQ_p))$, we add up all the cases.
\end{proof}

By \eqref{eq:convergence_factors} and Lemma~\ref{lem:p-adic_densities}, we have
\begin{equation}\label{eq:archimedean_factor}
  \prod_p \lambda_p^{-1}\omega_{\tH,p}(\tX(\QQ_p))  = \prod_p\left(1-\frac 1 p\right)^2\left(1+\frac 1 p+\frac 1{p^2}\right)^2 = \frac{1}{\zeta(3)^2}.
\end{equation}
Finally, we compute the real density.

\begin{lemma}\label{lem:real_density_+}
  We have
  \begin{equation*}
    \omega_\infty^{(1)} = \frac{3\vol(S_1^+(e))}{4}.
  \end{equation*}
\end{lemma}

\begin{proof}
  For $\xx = (x_0,x_1,x_2) \in \RR^{\oplus 3}$, we use the notation
  $|\xx| = \max\{|x_0|,|x_1|,|x_2|\}$. By definition (see
  Proposition~\ref{prop:asymptotic_formula}), we have
  \begin{equation*}
    \vol(S_1^+(e)) = \int_{F_1(\yy,\zz) \le 1,\ e^{-1} < |\yy|/|\zz| \le 1} \dd\yy \dd\zz.
  \end{equation*}
  By symmetry in $\zz$, this is
  \begin{equation*}
    = 6 \int_{F_1(\yy,\zz) \le 1,\ e^{-1} < |\yy|/z_0
      \le 1,\ |\zz|=z_0} \dd\yy \dd\zz.
  \end{equation*}
  For all $\lambda,\mu \in\RR$ and $\yy,\zz$, we observe that
  \begin{equation}\label{eq:F_1_homogeneous}
    F_1(\lambda \yy,\mu \zz) = |\lambda\mu|\cdot F_1(\yy,\zz).
  \end{equation}
  Using (\ref{eq:F_1_homogeneous}), the change of variables $z_1=z_0t_1$ and
  $z_2=z_0t_2$ gives
  \begin{equation*}
    = 6 \int_{z_0=0}^\infty z_0^2 \int_{|t_1|\le 1,\ |t_2| \le 1}
    \int_{z_0 F_1(\yy,1,t_1,t_2)\le 1,\ e^{-1} < |\yy|/z_0
      \le 1} \dd \yy \dd t_2 \dd t_1 \dd z_0.
  \end{equation*}
  The change of variables $y_i=z_0u_i$ gives
  \begin{equation*}
    = 6 \int_{z_0=0}^\infty z_0^5 \int_{|t_1|\le 1,\ |t_2| \le 1}
    \int_{z_0^2 F_1(\uu,1,t_1,t_2)\le 1,\ e^{-1}<|\uu|\le 1} \dd \uu \dd t_2 \dd t_1 \dd z_0.
  \end{equation*}
  The change of variables $w_0=z_0^2$, $w_1=z_0^2t_1$, $w_2=z_0^2t_2$ (with
  $\dd \ww = 2z_0^5 \dd z_0 \dd t_1 \dd t_2$) with (\ref{eq:F_1_homogeneous})
  gives
  \begin{equation*}
    = 3 \int_{w_0=0}^\infty \int_{|w_1|\le w_0,\ |w_2| \le w_0}
    \int_{F_1(\uu,\ww)\le 1,\ e^{-1}<|\uu|\le 1} \dd \uu \dd \ww.
  \end{equation*}
  By symmetry in $\ww$, this is
  \begin{equation*}
    = \frac{1}{2} 
    \int_{F_1(\uu,\ww)\le 1,\ e^{-1}<|\uu|\le 1} \dd \uu \dd \ww.    
  \end{equation*}
  Let $J(\uu):= \int_{F_1(\uu,\ww) \le 1} \dd \ww$.  Then
  \begin{equation*}
    \vol(S_1^+(e)) = \frac 1 2 \int_{e^{-1} < |\uu| \le 1} J(\uu) \dd \uu.
  \end{equation*}
  By symmetry, this is
  \begin{equation}\label{eq:J}
    =\frac 3 2 \int_{e^{-1} < |u_0| \le 1,\ |\uu| = |u_0|} J(\uu) \dd \uu.
  \end{equation}
  Using $J(c\uu) = |c|^{-3}J(\uu)$ for $c \in \RR_{\ne 0}$, this is
  \begin{equation*}
    =\frac 3 2 \int_{e^{-1} < |u_0| \le 1,\ |\uu| = |u_0|} |u_0|^{-3}J(1,u_1/u_0,u_2/u_0) \dd \uu.
  \end{equation*}
  The change of variables $u_i=u_0v_i$ for $i=1,2$ gives
  \begin{equation*}
    =\frac 3 2 \int_{|v_1|,|v_2| \le 1} J(1,v_1,v_2) \dd v_1 \dd v_2
    \int_{e^{-1} < |u_0| \le 1} |u_0|^{-1} \dd u_0,
  \end{equation*}
  where the integral over $u_0$ is clearly $2$. Hence
  \begin{equation}\label{eq:vol_2var}
    \vol(S_1^+(e)) = 3 \int_{|v_1|,|v_2| \le 1} J(1,v_1,v_2) \dd v_1 \dd v_2.
  \end{equation}

  Now we consider $\omega_\infty^{(1)}$, as defined in
  Lemma~\ref{lem:omega_v_parts}.  Using
  \begin{equation*}
    \frac{1}{s} = \frac 1 2 \int_{|t| \ge s} \frac{\dd t}{|t|^2} ,
  \end{equation*}
  we have
  \begin{equation*}
    \omega_\infty^{(1)} = \frac 1 4 \int_{F_1(1,y_1,y_2,1,z_1,z_2) \le |t|^{1/3}} |t|^{-2} \dd t\dd y_1 \dd y_2 \dd z_1
    \dd z_2.
  \end{equation*}
  The coordinate change $t=x_0^{-3}$ (with $|t|^{-2} \dd t = 3|x_0|^2 \dd
  x_0$) gives
  \begin{equation*}
    = \frac 3 4 \int_{|x_0|\cdot F_1(1,y_1,y_2,1,z_1,z_2) \le 1} |x_0|^2 \dd x_0 \dd y_1 \dd y_2 \dd z_1
      \dd z_2.
  \end{equation*}
  With (\ref{eq:F_1_homogeneous}), the coordinate change $y_j=x_j/x_0$ (with
  $|x_0| \dd y_j = \dd x_j$) for $j=1,2$ gives
  \begin{equation*}
    = \frac 3 4 \int_{F_1(\xx,1,z_1,z_2) \le 1} \dd \xx
    \dd z_1 \dd z_2.
  \end{equation*}
  With $J$ defined as above and by symmetry, this is
  \begin{equation*}
    = \frac 3 4 \int_{\RR^{\oplus 2}} J(1,z_1,z_2) \dd z_1 \dd z_2.
  \end{equation*}
  Using
  \begin{equation*}
    \frac 1 2 \int_{e^{-1} \le |v_0| \le 1} |v_0|^{-1} \dd v_0 = 1,
  \end{equation*}
  this is
  \begin{equation*}
    =\frac 3 8 \int_{e^{-1} \le |v_0| \le 1} |v_0|^{-1} J(1,z_1,z_2) \dd v_0
    \dd z_1 \dd z_2.
  \end{equation*}
  Using $|v_0|^{-1} J(1,z_1,z_2) = |v_0|^2 J(v_0,v_0z_1,v_0z_2)$, the
  coordinate change $z_j=v_j/v_0$ (with $\dd z_j = |v_0|^{-1} \dd v_j$) gives
  \begin{equation*}
    =\frac 3 8 \int_{e^{-1} \le |v_0| \le 1} J(\vv) \dd \vv.
  \end{equation*}
  We split this into three integrals depending on which $|v_i|$ is the
  largest, and we use symmetry for the cases $i=1,2$:
  \begin{equation}\label{eq:omega_infty_sum}
    = \frac 3 8 \int_{e^{-1}\le |v_0| \le 1,\ |\vv|=|v_0|} J(\vv) \dd \vv +
    \frac 3 4 \int_{e^{-1}\le |v_0| \le 1,\ |\vv|=|v_1|} J(\vv) \dd \vv.
  \end{equation}
  The first summand is clearly $\frac 1 4 \vol(S_1^+(e))$ (see
  (\ref{eq:J})). The second summand is
  \begin{equation*}
    \frac 3 4 \int_{e^{-1}\le |v_0| \le 1, |\vv|=|v_1|} |v_1|^{-3}
    J(v_0/v_1,1,v_2/v_1) \dd \yy
  \end{equation*}
  The change of variables $y_i=y_1v_i$ for $i=0,2$ gives
  \begin{equation*}
    =\frac 3 4 \int_{|v_0|,|v_2|\le 1} J(v_0,1,v_2)\left(\int_{e^{-1}\le
        |y_1v_0| \le 1} |y_1|^{-1} \dd y_1\right)\dd v_0 \dd v_2.
  \end{equation*}
  Here, the integral over $y_1$ is $2$, hence we have
  \begin{equation*}
    = \frac 3 2 \int_{|v_0|,|v_2|\le 1} J(v_0,1,v_2)\dd v_0 \dd v_2.
  \end{equation*}
  Hence the second summand in (\ref{eq:omega_infty_sum}) is
  $\frac 1 2 \vol(S_1^+(e))$ by symmetry and comparing to
  (\ref{eq:vol_2var}). In total,
  $\omega_\infty = \frac 3 4 \vol(S_1^+(e))$.
\end{proof}

\begin{lemma}\label{lem:real_density_-}
  We have
  \begin{equation*}
    \omega_\infty^{(-1)} = \frac{6\vol(S_1^-)}{\pi}.
  \end{equation*}
\end{lemma}

\begin{proof}
  By definition,
  \begin{equation*}
    \vol(S_1^-) =\vol\{\zz \in \CC^{\oplus 3} : F_{-1}(\zz) \le 1\} = \int_{F_{-1}(\zz) \le 1} \dd \zz.
  \end{equation*}
  In polar coordinates $z_j = \sqrt{r_j} e^{i\phi_j}$ (with
  \begin{equation*}
    |z_j|^2 = r_j, \quad|z_j\overline{z_k}+z_k\overline{z_j}| =
    \sqrt{r_jr_k}|e^{i(\phi_j-\phi_k)}+e^{i(\phi_k-\phi_j)}| =
    2\sqrt{r_jr_k}|\cos(\phi_j-\phi_k)|,
  \end{equation*}
  and $\dd z_j = \frac 1 2 \dd r_j \dd \phi_j$), we have
  \begin{equation*}
    \frac{6\vol(S_1^-)}{\pi} = \frac 3{4\pi} \int_{r_j \in \RR_{\ge 0},\ \phi_j \in [0,2\pi[,\ G(r_0,
      \dots, \phi_2)\le 1} \dd r_0 \dots \dd \phi_2,
  \end{equation*}
  where
  \begin{equation*}
    G(r_0,\dots, \phi_2):=\max\left\{
    \begin{aligned}
      &r_0,r_1,r_2,2\sqrt{r_0r_1}\cos(\phi_0-\phi_1),\\
      &2\sqrt{r_0r_2}\cos(\phi_0-\phi_2),2\sqrt{r_1r_2}\cos(\phi_1-\phi_2)
    \end{aligned}
    \right\}.
  \end{equation*}
  We make the change of coordinates $\phi_1'=\phi_1-\phi_0$,
  $\phi_2=\phi_2-\phi_0$ and obtain
  \begin{equation*}
    = \frac 3{4\pi} \int_{r_j \in \RR_{\ge 0},\ \phi_0,\phi_1',\phi_2' \in
      [0,2\pi[,\ G(r_0,r_1,r_2,0,\phi_1',\phi_2') \le 1} \dd r_0 \dd r_1 \dd
    r_2 \dd \phi_0 \dd \phi_1' \dd \phi_2'.
  \end{equation*}
  We perform the integration over $\phi_0$ and get
  \begin{equation*}
    = \frac{3}{2}\int_{r_j \in \RR_{\ge 0},\ \phi_1',\phi_2' \in
      [0,2\pi[,\ G(r_0,r_1,r_2,0,\phi_1',\phi_2') \le 1} \dd r_0 \dd r_1 \dd
    r_2 \dd \phi_1' \dd \phi_2'.
  \end{equation*}
  We make the change of coordinates $r_1=r_0r_1'$, $r_2=r_0r_2'$ with
  \begin{equation*}
    G(r_0,r_0r_1',r_0r_2',0,\phi_1',\phi_2') = r_0 \cdot
    G(1,r_1',r_2',0,\phi_1',\phi_2')
  \end{equation*}
  and obtain
  \begin{equation*}
    = \frac{3}{2}\int_{r_0,r_1',r_2' \in \RR_{\ge 0},\ \phi_1',\phi_2' \in
      [0,2\pi[,\ G(1,r_1',r_2',0,\phi_1',\phi_2') \le r_0^{-1}} r_0^2 \dd r_0 \dd r_1' \dd
    r_2' \dd \phi_1' \dd \phi_2'.
  \end{equation*}
  We make the change of coordinates $r_0 = t^{-1/3}$ (with $r_0^2 \dd r_0 =
  \frac 1 3 t^{-2} \dd t$) and obtain
  \begin{equation*}
    = \frac{1}{2} \int_{t,r_1',r_2' \ge 0,\ \phi_1',\phi_2' \in
      [0,2\pi[,\ G(1,r_1',r_2',0,\phi_1',\phi_2')^3 \le t} t^{-2} \dd t \dd
    r_1' \dd r_2' \dd \phi_1' \dd \phi_2'.
  \end{equation*}
  Using the identity
  \begin{equation*}
    \frac{1}{s} = \int_{t \ge s} \frac{\dd t}{t^2},
  \end{equation*}
  this is
  \begin{equation*}
     = \frac{1}{2} \int_{r_1',r_2' \ge 0,\ \phi_1',\phi_2' \in
      [0,2\pi[} \frac{\dd r_1' \dd r_2'
    \dd \phi_1' \dd \phi_2'}{G(1,r_1',r_2',0,\phi_1',\phi_2')^3} .
  \end{equation*}
  We turn this expression in polar coordinates back into a complex integral
  ($z_j' = \sqrt{r_j'}e^{i\phi_j'}$ with 
  $\dd z_j' = \frac 1 2 \dd r_j' \dd \phi_j'$), giving
  \begin{equation*}
    = 2 \int_{\CC^{\oplus 2}} \frac{\dd z_1' \dd z_2'}{F_{-1}(1,z_1',z_2')^3} .
  \end{equation*}
  By definition (see Lemma~\ref{lem:omega_v_parts}), this is
  $\omega_\infty^{(-1)}$.
\end{proof}

Lemma~\ref{lem:omega_v_parts}, Lemma~\ref{lem:real_density_+}, and
Lemma~\ref{lem:real_density_-} prove that
\begin{equation*}
  \omega_{\tH,\infty}(\tX(\RR)) = \omega_\infty^{(1)} + \omega_\infty^{(-1)} = \frac{3 \vol(S_1^+(e))}{4} + \frac{6 \vol(S_1^-)}{\pi}.
\end{equation*}
Combining this with \eqref{eq:limit} and
\eqref{eq:archimedean_factor}, the Tamagawa number
\eqref{def:tamagawa} is
\begin{equation*}
  \tau_\tH(\tX) = \frac{\omega_{\tH,\infty}(\tX(\RR))}{\zeta(3)^2} = \frac{\omega_\infty^{(1)} + \omega_\infty^{(-1)}}{\zeta(3)^2} = \frac{3 \vol(S_1^+(e))}{4\zeta(3)^2} + \frac{6 \vol(S_1^-)}{\pi\zeta(3)^2}.
\end{equation*}
Together with \eqref{eq:alpha}, \eqref{eq:beta}, this shows that
Peyre's constant $\theta_\tH(\tX)$ as in \eqref{eq:peyres_constant} is
the sum of $c_+$ and $c_-$ as in
Proposition~\ref{prop:asymptotic_formula}. This completes the proof of
Theorem~\ref{thm:manin-peyre}.

\bibliographystyle{alpha}

\bibliography{manin_sym-4}

\end{document}